\documentclass{article}
\usepackage[letterpaper, portrait, margin=2.5cm]{geometry}

\usepackage{verbatim}
\usepackage[utf8]{inputenc}
\usepackage{amsmath,amssymb,amsthm}
\usepackage{tikz-cd}
\usepackage{bbm}

\usepackage{comment}
\usepackage{stmaryrd}
\usepackage[new]{old-arrows}
\usepackage{graphicx}
\graphicspath{ {./Images/} }
\usepackage{todonotes}

\usepackage{multirow}

\usepackage [english]{babel}
\usepackage [autostyle, english = american]{csquotes}
\MakeOuterQuote{"}

\newcommand{\catname}[1]{{\normalfont\textbf{#1}}}

\newcommand{\Set}{\catname{Set}}

\newcommand{\RR}{\mathbb R}
\newcommand{\ZZ}{\mathbb Z}

\newcommand{\CC}{\mathbb C}

\newcommand{\PP}{\mathbb P}
\newcommand{\VV}{\mathbb V}

\newcommand{\AAA}{\mathbb A}

\newcommand{\mcc}{\mathcal C}

\newcommand{\mch}{\mathcal H}

\newcommand{\mcm}{\mathcal M}

\newcommand{\mcs}{\mathcal S}
\newcommand{\mct}{\mathcal T}

\newcommand{\mfc}{\mathfrak C}

\newcommand{\lbm}{\left[ \begin{matrix}}
\newcommand{\rem}{\end{matrix} \right]}
\newcommand{\SL}{\sum\limits}
\newcommand{\PL}{\prod\limits}

\newcommand{\ra}{\rightarrow}

\newcommand{\Hom}{\text{Hom}}

\newcommand{\lp}{\left(}
\newcommand{\rp}{\right)}
\newcommand{\lb}{\left\{}
\newcommand{\rb}{\right\}}

\newcommand{\lav}{\left|}
\newcommand{\rav}{\right|}
\newcommand{\inv}{^{-1}}
\newcommand{\st}{ \ | \ }

\newcommand{\wt}{\widetilde} 

\newcommand{\proj}{\text{Proj}}
\newcommand{\p}{\partial}

\newcommand{\RD}{\text{RD}}

\newcommand{\trdeg}{\text{tr.deg}}
\newcommand{\spec}{\text{Spec}}

\newcommand{\gr}{\text{Gr}}
\newcommand{\GR}{\mathcal{G}\text{r}}

\newcommand{\PGL}{\text{PGL}}
\newcommand{\GL}{\text{GL}}

\theoremstyle{definition}
\newtheorem{definition}{Definition}[section]
\newtheorem{remark}[definition]{Remark}
\newtheorem{example}[definition]{Example}

\newtheorem{prop}[definition]{Proposition}
\newtheorem{lemma}[definition]{Lemma}
\newtheorem{theorem}[definition]{Theorem}
\newtheorem{corollary}[definition]{Corollary}
\newtheorem{problem}[definition]{Problem}

\title{Upper Bounds on Resolvent Degree and Its Growth Rate}
\author{Alexander J. Sutherland\footnote{This work was supported in part by the National Science Foundation under Grant No. DMS-1944862.}}
\date{\today}

\begin{document}

\maketitle

\begin{abstract}
	For each $n$, let $\RD(n)$ denote the minimum $d$ for which there exists a formula for the general polynomial of degree $n$ in algebraic functions of at most $d$ variables. In 1945, Segre called for a better understanding of the large $n$ behavior of $\RD(n)$. In this paper, we provide improved thresholds for upper bounds on $\RD(n)$. Our techniques build upon classical algebraic geometry to provide new upper bounds for small $n$ and, in doing so, fix gaps in the proofs of A. Wiman and G.N. Chebotarev in \cite{Wiman1927} and \cite{Chebotarev1954}. 
\end{abstract}

\tableofcontents

\section{Introduction}\label{sec:Introduction}

Consider the following classical problem:

\begin{problem}\label{prob:Determine Fomulas for Generic Polynomials}
 Given a general polynomial $z^n + a_1z^{n-1} + \cdots + a_{n-1}z + a_n$, determine a root of the polynomial in terms of $a_1,\dotsc,a_n$ in the simplest manner possible.
\end{problem}

Recent work on this problem has been cast within the framework of resolvent degree (as in \cite{FarbWolfson2019}). Informally, $\RD(n)$ is the minimal $d$ for which there exists a formula for the roots of the generic degree $n$ polynomial using algebraic functions of at most $d$ variables (see Definitions \ref{def:ResolventDegreeFieldExtensions} and \ref{def:ResolventDegreeBranchedCovers}). While there are currently no non-trivial lower bounds on $\RD(n)$ and it is possible in theory that $\RD(n)=1$ for all $n$, there is a history of determining upper bounds on $\RD(n)$. This includes the work of Bring \cite{Bring1786}, Hamilton \cite{Hamilton1836}, Sylvester \cite{Sylvester1887,SylvesterHammond1887,SylvesterHammond1888}, Klein \cite{Klein1884,Klein1887,Klein1905}, and Hilbert (\cite{Hilbert1927}).\footnote{Chen, He, and McKay provide a modern English translation of \cite{Bring1786} (originally in Latin) in \cite{ChenHeMcKay2017} and the author provides a modern English translation of \cite{Klein1905} (originally in German) in \cite{Sutherland2019}.} Indeed, Hamilton proved that $\RD(6) \leq 2$, $\RD(7) \leq 3$, and $\RD(8) \leq 4$. Hilbert's sextic conjecture, Hilbert's 13th problem, and Hilbert's octic conjecture, respectively, predict that these upper bounds are sharp. In Section 3 of \cite{Segre1945}, Segre indicated the following problem:

\begin{problem}\label{prob:Segre}  \textbf{(Segre, 1945)}\\
	Determine the large $n$ behavior of $\RD(n)$.
\end{problem}

The current upper bounds on $\RD(n)$ are found in \cite{Wolfson2021}, where Wolfson introduces a function $F(m)$ such that $\RD(n) \leq n-m$ for all $n \geq F(m)$ (Definition 5.4 and Theorem 5.6, \cite{Wolfson2021}). In this paper, we construct a similar function $G(m)$. For $21 \leq n \leq 3,632,428,800$, we use methods which build upon the work of Segre. By using a theorem of Debarre and Manivel (see Theorem \ref{thm:DebarreManivel}, or Theorem 2.1 of \cite{DebarreManivel1998}), we are able to streamline Wolfson's method to obtain better thresholds on upper bounds on $\RD(n)$ for $n \geq 348,489,068,134$. For $3,632,428,801 \leq n \leq 348,489,068,133$, we obtain the same bounds as Wolfson's method. More specifically, we prove the following results:

\begin{theorem}\label{thm:Key Properties of G(m)} \textbf{(Key Properties of $G(m)$)}\\
	The function $G(m)$ of Definition \ref{def:Definition of G} has the following properties:
	\begin{enumerate}
		\item For each $m \geq 1$ and $n \geq G(m)$, $\RD(n) \leq n-m$.
		\item For each $d \geq 4$, $G(2d^2+7d+6) \leq \frac{(2d^2+7d+5)!}{d!}$. In particular, for $d \geq 4$ and $n \geq \frac{(2d^2+7d+5)!}{d!}$,
		\begin{equation*}
			\RD(n) \leq n-2d^2-7d-6.
		\end{equation*}
		\item For each $m \geq 1$, $G(m) \leq F(m)$ with equality only when $m \in \lb 1,2,3,4,5,15,16 \rb$ and
		\begin{equation*}
			\lim\limits_{m \ra \infty} \frac{F(m)}{G(m)} = \infty.
		\end{equation*}		 
	\end{enumerate}	
\end{theorem}

The first statement of Theorem \ref{thm:Key Properties of G(m)} is Theorem \ref{thm:UpperBoundsonRD(n)} and, in light of this statement, we refer to $F(m)$ and $G(m)$ as ``bounding functions'' for $\RD(n)$. Hamilton was the first to provide a bounding function for $\RD(n)$ (see \cite{OEIS2021}). Brauer greatly improved Hamilton's result when he showed that $\RD(n) \leq n-m$ for $n \geq (m-1)!+1$ in \cite{Brauer1975}. Wolfson improved upon these bounds with his function $F(m)$, but did not provide an upper bound on $F(m)$ in terms of elementary functions. The construction of $G(m)$ is simpler than $F(m)$, but is not as explicit as Brauer's condition. The second statement of Theorem \ref{thm:Key Properties of G(m)} gives an upper bound on $\RD(n)$ and its growth rate using elementary functions; it is Theorem \ref{thm:UpperBoundOnGrowthRateofRD(n)}. 

The third statement of Theorem \ref{thm:Key Properties of G(m)} shows that $G(m)$ provides better asymptotic bounds than $F(m)$; it is Theorem \ref{thm:ComparisonWithF}. Note that the inequalities when $1 \leq m \leq 4$ (which imply $\RD(n)=1$ for $2 \leq n \leq 5$) are due to the classical solutions of general polynomials of low degree. Recall that Hilbert's Octic Conjecture predicts that $\RD(8)=4$ and, if true, would imply that the bounds $G(5) = F(5) = 9$ cannot be improved. We expect that the general thresholds obtained by $G(m)$ are not sharp (including the cases of $G(15)=F(15)$ and $G(16)=F(16)$).

\paragraph{Historical Remarks}
Many of the approaches above rely on the theory of Tschirnhaus transformations. In \cite{Wolfson2021}, Wolfson consolidates the history of Tschirnhaus transformations and establishes upper bounds on $\RD(n)$ by determining points of Tschirnhaus complete intersections over field extensions of bounded resolvent degree. In \cite{Wiman1927}, Wiman used these methods to give an argument that implied $\RD(n) \leq n-5$ for $n \geq 9$ and Chebotarev\footnote{Note that this is G.N. Chebotarev, son of Nikolai Chebotarev.} extended Wiman's idea to argue that $\RD(n) \leq n-6$ for $n \geq 21$ in \cite{Chebotarev1954}. Dixmier noted in \cite{Dixmier1993}, however, that Wiman's argument has a gap and provided an algebraic argument to conclude that $\RD(n) \leq n-5$ for $n \geq 9$. Moreover, the gap found in Wiman's argument is also present in the argument of Chebotarev.

Building on the work of Segre, we introduce (iterated) polar cones for hypersurfaces and intersections of hypersurfaces, which we use to determine linear subvarieties on intersections of hypersurfaces over extensions of bounded resolvent degree. In doing so, we fix the gaps in the arguments of Chebotarev and provide a geometric fix for the arguments of Wiman.

\paragraph{Outline of the Paper}

In Section \ref{sec:PolarCones}, we use the language of polar cones (respectively, iterated polar cones) and highlight their connection to lines (respectively, $k$-planes) on intersections of hypersurfaces. In Section \ref{sec:NewUpperBoundsOnRD(n)}, we use techniques involving iterated polar cones and moduli spaces to construct a function $G(m)$ which yields the upper bounds on $\RD(n)$. In Section \ref{sec:ComparisonWithPriorBounds}, we compare these new bounds to those of Wolfson. In Appendix \ref{subsec:Appendix - Explicit Bounds}, we give explicit values of $G(m)$ for small values of $m$. In Appendix \ref{subsec:Appendix - Explicit Approximations of F(m)/G(m)}, we provide additional data on the ratio $F(m)/G(m)$. In Appendix \ref{subsec:Appendix-ProofOfTechnicalLemma}, we provide the proof of a technical lemma from Section \ref{sec:PolarCones}.

\paragraph{Conventions}

\begin{enumerate}
\item We restrict to fields $K$ which are finitely generated $\CC$-algebras. The interested reader could instead fix an arbitrary algebraically closed field $F$ of characteristic zero (in lieu of $\CC$) and the statements (relative to $F$) would hold.

\item For varieties, we follow the conventions of \cite{Harris2010}. Namely, we define a projective (respectively, affine) variety to be a closed algebraic set in $\PP_K^n$ (respectively, $\AAA_K^n$). When we simply say variety, we mean a quasi-projective variety. In particular, we do not assume that varieties are irreducible.

\item For a collection of homogeneous polynomials $f_1,\dotsc,f_s \in K[x_0,\dotsc,x_r]$, we write $\VV(f_1,\dotsc,f_s)$ for the subvariety of $\PP_K^r$ determined by the conditions $f_1 = \cdots = f_s = 0$.

\item Given a subvariety $V \subseteq \PP_K^r$, we use the notation $V(K)$ to refer to the set of $K$-rational points of $V$.

\item We write $K_n$ to mean $\CC(a_1,\dotsc,a_n)$, a purely transcendental extension of $\CC$ with transcendence basis $a_1,\dotsc,a_n$.

\item We write $\gr(k,r)$ to denote the Grassmannian of $k$-dimensional subspaces of $\CC^r$ and $\GR(k,r)$ for the space of $k$-planes in $\PP_\CC^r$.  In particular, $\GR(k,r) \cong \gr(k+1,r+1)$.

\item Given a polynomial ring $K[x_0,\dotsc,x_r]$ over a field $K$, we write
\begin{itemize}
\item $K[x_0,\dotsc,x_r]_{(d)}$ for the vector space of degree $d$ polynomials,
\item $K[x_0,\dotsc,x_r]_{(d)}^\vee$ for its dual space,
\item $S^*\lp K[x_0,\dotsc,x_r]_{(d)}^\vee \rp$ for the corresponding free commutative $K$-algebra, and
\item $S^*\lp K[x_0,\dotsc,x_r]_{(d)}^\vee \rp^{\GL(K,r+1)}$ for the associated graded ring of $\GL(K,r+1)$-invariants.
\end{itemize}

\item We write $\log$ to mean the base $e$ logarithm.
\end{enumerate}

With respect to convention 2, note that for generic choices of $f_1,\dotsc,f_s$, $\VV(f_1,\dotsc,f_s)$ is a complete intersection. However, there are examples of such choices which are not complete intersections, such as the twisted cubic curve. Our methods apply to all such choices and to avoid confusion, we generally use the terminology ``intersection of hypersurfaces'' instead of ``complete intersection'' to refer to a variety of the form $\VV(f_1,\dotsc,f_s)$. Additionally, we note that conventions 5 and 6 are primarily for Definition \ref{def:ParameterAndModuliSpacesOfHypersurfacesWithkPlanes} and Remark \ref{rem:Parameter and Moduli Spaces as Schemes}, respectively.

\paragraph{Acknowledgements}

I thank Jesse Wolfson for his generous support. Next, I thank Joshua Jordan for several key conversations and general support. I thank Benson Farb, Hannah Knight, Curt McMullen, and Zinovy Reichstein for helpful comments on a draft. Finally, I thank Kenneth Ascher, Claudio G\'{o}mez-Gonz\'{a}les, and Roman Vershynin for helpful conversations.


\section{Polar Cones}\label{sec:PolarCones}

\subsection{An Introduction to Polar Cones}\label{subsec:An Introduction to Polar Cones}

In \cite{Segre1945}, Segre refers to a ``well-known fact'' which is the main ingredient for many of his proofs. This fact may no longer be as well-known as it once was. Definitions \ref{def:PolarsOfAPolynomial}, \ref{def:PolarsOfAHypersurface}, and \ref{def:PolarConeOfAHypersurface} are necessary to state the fact, which we include as Lemma \ref{lem:BertinisLemmaForHypersurfaces} and refer to as Bertini's Lemma (for Hypersurfaces), since the reference Segre gives for this fact is \cite{Bertini1923}.

\begin{definition}\label{def:PolarsOfAPolynomial} \textbf{(Polars of a Polynomial)}\\
 Let $f \in K[x_0,\dotsc,x_r]$ be a homogeneous polynomial of degree $d$ and $P \in \PP^r(K)$. For every $n \in \ZZ_{\geq 1}$, let
 \begin{align*}
 	&[n] = \lb 0,1,\dotsc,n \rb,\\
 	&[n]^* = \lb 1,\dotsc,n \rb.
\end{align*}

\noindent
For each $0 \leq k \leq d-1$, we set
\begin{align*}
	I_k := \lb (i_0,\dotsc,i_r) \in \ZZ_{\geq 0}^{r+1} \st i_0+\cdots+i_r=k \rb.
\end{align*}

We have $r+1$ first partial derivatives $\frac{\p}{\p x_0},\dotsc,\frac{\p}{\p x_r}$. To obtain an (ordered) $k^{th}$ partial derivative of $f$, we make $k$ successive choices of first partial derivatives. Hence, the set of all (ordered) $k^{th}$ partial derivatives is indexed by $I_k^* = \Hom_{\Set}([k]^*,[r])$. For any $\iota \in I_k^*$ and $0 \leq j \leq r$, we set
\begin{align*}
	|\iota|(j) := \lav \iota\inv(j) \rav
\end{align*}

\noindent
and note that $(|\iota|(0),\dotsc,|\iota|(r)) \in I_k$. Additionally, we use the shorthand
\begin{align*}
	 \p_0^{j_0} \cdots \p_l^{j_l} f := \frac{\p^{j_0+\cdots+j_l} f}{\p x_0^{j_0} \cdots \p x_l^{j_l}}.
\end{align*}

\noindent
With this notation set, the \textbf{$k^{th}$ polar of $f$ at $P$} is the homogeneous polynomial of degree $d-k$
 \begin{align*}
 	t(k,f,P)(y_0,\dotsc,y_r) := \SL_{\iota \in I_{d-k}^*} \lp \p_0^{|\iota|(0)} \cdots \p_r^{|\iota|(r)} f \rp \biggr\rvert_P \  y_0^{|\iota|(0)}\cdots y_r^{|\iota(r)|}.
 \end{align*}
\end{definition}

\begin{definition}\label{def:PolarsOfAHypersurface} \textbf{(Polars of a Hypersurface)}\\
Let $V=\VV(f)$ be a hypersurface in $\PP_K^r$ and $P \in V(K)$. The \textbf{$k^{th}$ polar of $V$ at $P$} is
\begin{align*}
	T(k,f,P) := \VV\lp t(k,f,P) \rp \subseteq \PP_K^r.
\end{align*}
\end{definition}

\begin{example}\label{rem:ExamplesOfPolars}
Using the conventions of Definition \ref{def:PolarsOfAHypersurface}, we observe that
\begin{align*}
	t(0,f,P)(x_0,\dotsc,x_r) &= f(x_0,\dotsc,x_r),\\
	t(d,f,P)(x_0,\dotsc,x_r) &= f(P),
\end{align*}

\noindent
so the $0^{th}$ polar of $\VV(f)$ at $P$ is $T(0,f,P) = V$ for all $P$ and the $d^{th}$ polar of $\VV(f)$ at $P$ is
\begin{align*}
	T(d,f,P) =
	\begin{cases}
	\PP_K^r, &\text{if } P \in \VV(f),\\
	\emptyset, &\text{if } P \not\in \VV(f).
	\end{cases}
\end{align*}

Moreover, if $V$ is smooth at $P$, then $T(d-1,f,P)$ is the tangent hyperplane of $V$ at $P$; this motivates the use of $T(k,f,P)$ for polar hypersurfaces and hence $t(k,f,P)$ for their defining polynomials.
\end{example}

\begin{definition}\label{def:PolarConeOfAHypersurface} \textbf{(Polar Cone of a Hypersurface)}\\
Let $V = \VV(f)$ be a degree $d$ hypersurface in $\PP_K^r$ and $P \in V(K)$. The \textbf{(first) polar cone} of $V$ at $P$ is
\begin{align*}
	\mcc(V;P) := \bigcap\limits_{k=0}^{d-1} T(k,f,P).
\end{align*}
\end{definition}

We can now state Lemma \ref{lem:BertinisLemmaForHypersurfaces}, which motivates the terminology of the previous definition.

\begin{lemma}\label{lem:BertinisLemmaForHypersurfaces} \textbf{(Bertini's Lemma for Hypersurfaces)}\footnote{See Chapter 8 of \cite{Bertini1923} for the reference given in \cite{Segre1945}.}\\
	Let $V = \VV(f)$ be a hypersurface in $\PP_K^r$ and $P \in V(K)$. Then, $\mcc(V;P)$ is a cone with vertex $P$ which is contained in $V$.
\end{lemma}

\begin{example}\label{ex:PolarConeForCubicSurfaces} \textbf{(Lines on Cubic Surfaces)}\\
Let $V \subseteq \PP_K^3$ be a smooth cubic surface. If $K=\overline{K}$, then $V$ contains exactly 27 lines $L_1,\dotsc,L_{27}$. If $P \in V(K)$ lies on exactly one line $L$ (respectively, exactly two lines $L_1, L_2$), then $\mcc(V;P) = L$ (respectively, $\mcc(V;P) = L_1 \cup L_2$). Most points $Q \in V(K)$, however, do not lie on any line and in such a case $\mcc(V;Q)$ is the point $Q$ with multiplicity 6. We will generally consider fields $K$ which are not algebraically closed, in which case we may need to pass to an extension of $K$ to obtain a line. 
\end{example}

\begin{remark}
	For an alternative modern perspective on Definitions \ref{def:PolarsOfAPolynomial} and \ref{def:PolarsOfAHypersurface}, Remark \ref{rem:ExamplesOfPolars}, and Lemma \ref{lem:TechnicalLemma}, see  p.5-6 of \cite{Dolgachev2012} - in particular, equations 1.8 through 1.11. Note that Lemma \ref{lem:BertinisLemmaForHypersurfaces} follows directly from the following technical lemma. The proof of Lemma \ref{lem:TechnicalLemma} can be found in Appendix \ref{subsec:Appendix-ProofOfTechnicalLemma}; it is not exceedingly complicated, but is notationally cumbersome.	
\end{remark}

\begin{lemma}\label{lem:TechnicalLemma} \textbf{(Technical Lemma)}\\
	Let $P,Q \in \PP^r(K)$ and $f \in K[x_0,\dotsc,x_r]$ be a homogeneous polynomial of degree $d$. Applying a projective change of coordinates as necessary, we assume that 
	\begin{align*}
		P &= [1:p_1:\cdots:p_r],\\
		Q &= [1:q_1:\cdots:q_r],
	\end{align*}		
	
	\noindent
	so that the line determined by $P$ and $Q$ is
	\begin{align*}
		L(P,Q)(K) = \lb [1:\lambda p_1 + \mu q_1:\cdots:\lambda p_r + \mu q_r] \st [\lambda:\mu] \in \PP^1(K) \rb.
	\end{align*}
	
	\noindent
	For any point $R_{\lambda:\mu} = [1:\lambda p_1 + \mu q_1:\cdots:\lambda p_r + \mu q_r] \in L(P,Q)(K)$,
	\begin{equation}\label{eq:TechnicalLemmaEquation}
		f\lp R_{\lambda:\mu} \rp = f(\lambda P) + f(\mu Q) + \SL_{k=1}^{d-1} \frac{1}{k!} t(d-k,f,\lambda P)(\mu Q).
	\end{equation}
\end{lemma}

\begin{remark}
	Lemma \ref{lem:BertinisLemmaForHypersurfaces} implies that for every point $Q \in \mcc(V;P)(K) \setminus \{P\}$, the line $L(P,Q)$ determined by $P$ and $Q$ lies in $V$. Furthermore, a line $L$ lies on an intersection of hypersurfaces $\VV(f_1,\dotsc,f_n)$ exactly when it lies on each $\VV(f_i)$. This motivates Definition \ref{def:Polar Cone Of An Intersection of Hypersurfaces} and implies Lemma \ref{lem:Bertinis Lemma For Intersections Of Hypersurfaces}.
\end{remark}

\begin{definition}\label{def:Polar Cone Of An Intersection of Hypersurfaces} \textbf{(Polar Cone of an Intersection of Hypersurfaces)}\\
	Let $V = \VV(f_1,\dotsc,f_n) \subseteq \PP_K^r$ be an intersection of hypersurfaces and $P \in V(K)$. The \textbf{(first) polar cone of $V$ at $P$} is
	\begin{align*}
		\mcc(V;P) := \bigcap\limits_{i=1}^n \mcc(\VV(f_i);P).
	\end{align*}
\end{definition}

\begin{lemma}\label{lem:Bertinis Lemma For Intersections Of Hypersurfaces} \textbf{(Bertini's Lemma for Intersections of Hypersurfaces)}\\
	Let $V = \VV(f_1,\dotsc,f_n) \subseteq \PP_K^r$ be an intersection of hypersurfaces and $P \in V(K)$.  Then, $\mcc(V;P)$ is a cone with vertex $P$ which is contained in $V$.
\end{lemma}

\begin{remark}\label{rem:Passing to Extensions} \textbf{(Passing to Extensions)}\\
	We continue using the notation of Lemma \ref{lem:Bertinis Lemma For Intersections Of Hypersurfaces}. It is easy to see that when $K$ is algebraically closed and $\dim\lp \mcc(V;P) \rp \geq 1$, the set $\mcc(V;P)(K) \setminus \{P\}$ is not empty. Thus, any choice of $Q \in \mcc(V;P) \setminus \{P\}$ determines a line on $V$ containing $P$. When $K$ is not algebraically closed, however, it may be that $\mcc(V;P)(K) = \{P\}$. In what follows, we will often extend scalars to an extension $L/K$ over which we can determine an $L$-rational point.
\end{remark}

\subsection{Resolvent Degree}\label{subsec:Resolvent Degree}

Resolvent degree is an invariant introduced independently by Brauer in \cite{Brauer1975} and Arnol'd-Shimura in \cite{ArnoldShimura1976}. Farb and Wolfson summarize the history of resolvent degree and significantly broaden the resolvent degree framework in \cite{FarbWolfson2019}, which we refer the reader to for background on resolvent degree. We now recall the definition of resolvent degree for finite field extensions and for generically finite, dominant, rational maps of $\CC$-varieties. We refer the reader to Section 4 of \cite{Wolfson2021} for how to extend this definition to arbitrary dominant, rational maps. We also remind the reader that in this paper, we work only with fields that are finitely generated $\CC$-algebras.

\begin{definition}\label{def:ResolventDegreeFieldExtensions} \textbf{(Resolvent Degree of Field Extensions)}\\
	Let $L/K$ be a finite extension of $\CC$-fields. The \textbf{resolvent degree} of $L/K$, denoted $\RD(L/K)$, is the minimum $d$ for which there is a tower of finite field extensions
	\begin{align*}
		K = E_0 \hookrightarrow E_1 \hookrightarrow \cdots \hookrightarrow E_\ell
	\end{align*}
	
	\noindent
	such that $L$ embeds into $E_\ell$ over $K$ and for each $E_{i+1}/E_i$, there is a 
	finite extension of $\wt{F}_i / F_i$ with $\trdeg_\CC(F_i) \leq d$ such that
	\begin{align*}
		E_{i+1} \cong E_i \otimes_{F_i} \wt{F}_i.
	\end{align*}
\end{definition}

\begin{definition}\label{def:ResolventDegreeBranchedCovers} \textbf{(Resolvent Degree of Generically Finite, Dominant Maps)}\\
	Let $Y \dashrightarrow X$ be a generically finite, dominant, rational map of $\CC$-varieties. The \textbf{resolvent degree} of $Y \dashrightarrow X$, denoted $\RD(Y \dashrightarrow X)$, is the minimum $d$ for which there is a tower of generically finite, dominant, rational maps
	\begin{align*}
		E_\ell \dashrightarrow \cdots \dashrightarrow E_1 \dashrightarrow E_0 \subseteq X
	\end{align*}
	
	\noindent
	such that $E_0 \subseteq X$ is a dense Zariski open; $E_\ell \dashrightarrow E_0$ factors through $Y \dashrightarrow X$; and for each $\pi_i:E_{i+1} \dashrightarrow E_i$, there exists a surjective morphism $\wt{Z}_i \ra Z_i$ with $\dim\lp Z_i \rp \leq d$, a Zariski open $E_i^\circ \subseteq E_i$, and a morphism $E_i^\circ \ra Z_i$ such that
	\begin{align*}
		\pi_i\inv\lp E_i^\circ \rp \cong E_i^\circ \times_{Z_i} \wt{Z}_i.
	\end{align*}
\end{definition}

\begin{remark}\label{rem:Compatibility of Resolvent Degree Definitions} \textbf{(Compatibility of Resolvent Degree Definitions)}\\
	We note that Definitions \ref{def:ResolventDegreeFieldExtensions} and \ref{def:ResolventDegreeBranchedCovers} are equivalent. For an irreducible, affine $\CC$-variety $X$, the equivalence of definitions is derived from the classical equivalence which sends $X$ to its function field $\CC(X)$; the general case follows from invariance of resolvent degree under birational equivalence.  We refer the reader to \cite{FarbWolfson2019} for details.
\end{remark}

\begin{remark}\label{rem:RD(n)Notation} \textbf{($\RD(n)$ Notation)}\\
	We write $\RD(n)$ for the resolvent degree of the general degree $n$ polynomial. Recall that $K_n = \CC(a_1,\dotsc,a_n)$. If $f(z) = z^n + a_1z^{n-1} + \cdots + a_{n-1}z + a_n \in K_n[z]$, and $K_n' = K_n[z]/(f(z))$, then
	\begin{align*}
	\RD(n) = \RD(K_n'/K_n) = \RD\lp \spec\lp K_n' \rp \ra \spec\lp K_n \rp \rp.
	\end{align*}
\end{remark}

\begin{remark}\label{rem:RD of a Composition} \textbf{(RD of a Composition)}\\
	We will frequently use the following field-theoretic version of Lemma 2.7 of \cite{FarbWolfson2019} without explicit reference: given a tower of field extensions
	\begin{align*}
		E_0 \hookrightarrow E_1 \hookrightarrow \cdots E_k,
	\end{align*}
	
	\noindent
	the resolvent degree of the composition is the maximum of the resolvent degree of the components, e.g.
	\begin{align*}
		\RD(E_k/E_0) = \max\lb \RD(E_j/E_{j-1}) \st 1 \leq j \leq k \rb.
	\end{align*}
\end{remark}

\begin{lemma}\label{lem:Upper Bound on L/K} \textbf{(Upper Bound on $\RD(L/K)$)}\\
	Let $L/K$ be a degree $\ell$ field extension. Then, $\RD\lp L/K \rp \leq \RD(\ell)$.
\end{lemma}

\begin{proof}
	As $L$ and $K$ are $\CC$-fields, they have characteristic zero. Hence, the Primitive Element Theorem yields that we need only solve a degree $\ell$ polynomial to determine a primitive element $\zeta$ of $L/K$. The isomorphism $L \cong K(\zeta)$ then establishes the claim.
\end{proof}

We will frequently make use of the previous lemma without explicit reference after giving an upper bound on the degree of an extension.

\begin{prop}\label{prop:RationalPointsOverExtensions} \textbf{(Determining Rational Points over Extensions)}\\
	Let $V \subseteq \PP_K^r$ be a degree $d$ subvariety. Then, there is an extension $L/K$ with $\RD(L/K) \leq \RD(d)$ over which we can determine a rational point of $V$ .
\end{prop}

\begin{proof}
	Set $\ell=\dim(V)$. For a generic $(r-\ell)$-plane $\Lambda$, the intersection $V \cap \Lambda$ has dimension 0 and thus has $d$ $\overline{K}$-points (with multiplicity) in any algebraic closure $\overline{K}$ of $K$; we denote these points by $Q_1,\dotsc,Q_d$. Observe that the polynomial
	\begin{align*}
		f(z) = \lp z-Q_1 \rp \lp z-Q_2 \rp \cdots \lp z-Q_d \rp
	\end{align*}	 
	
	\noindent
	has coefficients defined over $K$. Let $m(z)$ be an irreducible factor of $f(z)$ over $K$. We set $L := K[z]/(m(z))$ and observe that we can determine an $L$-point of $V$ by construction. Lemma \ref{lem:Upper Bound on L/K} yields that $\RD(L/K) \leq \RD(d)$.
\end{proof}

\begin{remark}\label{rem:SolvingPolynomials} \textbf{(Extensions Given By Solving Polynomials)}\\
	Given Proposition \ref{prop:RationalPointsOverExtensions}, we henceforth say that we can determine a point of $V$ by solving a degree $d$ polynomial.
\end{remark}

\subsection{Iterated Polar Cones and $k$-Polar Points}

	As we have established, for an intersection of hypersurfaces $V$ and $P \in V(K)$, points of $\mcc(V;P) \setminus \{P\}$ determine lines on $V$. Next, we will iterate the polar cone construction in a method to determine $k$-planes on intersections of hypersurfaces (as in Lemma \ref{lem:PolarPointLemma}). To do so, we must first introduce additional definitions and notation.

\begin{remark}\label{rem:NotationForLinearSubvarieties}  \textbf{(Linear Subvarieties)}
\begin{itemize}
	\item For any points $P_0,\dotsc,P_k \in \PP^r(K)$, we denote the linear subsvariety they determine by $L(P_0,\dotsc,P_k)$. Note that $L(P_0,\dotsc,P_k)$ has dimension at most $k$.
	\item Henceforth, when we say $k$-plane, we mean a linear subvariety of $\PP^r$ of dimension $k$. When $k=1$, we say line instead of 1-plane and when $k=2$, we say plane instead of 2-plane.
\end{itemize}
\end{remark}

\begin{definition}\label{def:IteratedPolarCones} \textbf{(Iterated Polar Cones \& $k$-Polar Points)}\\
	Let $V \subseteq \PP_K^r$ be an intersection of hypersurfaces and $P_0 \in V(K)$; we set $\mcc^1(V;P_0) := \mcc(V;P_0)$. Given additional points $P_1,\dotsc,P_{k-1} \in V(K)$ such that
	\begin{equation}\label{eqn:PolarPointsEquation}
	P_{\ell-1} \in \mcc^{\ell-1}(V;P_0,\dotsc,P_{\ell-2}) \setminus L(P_0,\dotsc,P_{\ell-2})
	\end{equation}
	
	\noindent
	for $0 \leq \ell \leq k-1$, the \textbf{$k^{th}$ polar cone of $V$ at $P_0,\dotsc,P_{k-1}$} is
	\begin{align*}
		\mcc^k(V;P_0,\dotsc,P_{k-1}) := \mcc\lp \mcc^{k-1}(V;P_0,\dotsc,P_{k-2}),P_{k-1} \rp.
	\end{align*}
	
	We refer to an ordered collection of points $\lp P_0,\dotsc,P_k \rp$ that satisfy (\ref{eqn:PolarPointsEquation}) for each $0 \leq \ell \leq k-1$ as a \textbf{$k$-polar point} of $V$. \\
	
	When the points in question have already been specified, we simply refer to \emph{the} $k^{th}$ polar cone of $V$. When such points exist but have not been specified, we refer to \emph{a} $k^{th}$ polar cone of $V$. It is occasionally useful to refer to $V$ itself as a zeroth polar cone of $V$ (at any rational point of $V$).
\end{definition}

\begin{remark}\label{rem:Iterated Polar Cones Are Nested} \textbf{(Iterated Polar Cones Are Nested)}\\
Note that if $\lp P_0,\dotsc,P_k \rp$ is a $k$-polar point of an intersection of hypersurfaces $V \subseteq \PP_K^r$, then
\begin{align*}
	\mcc^k(V;P_0,\dotsc,P_{k-1}) \subseteq \mcc^{k-1}(V;P_0,\dotsc,P_{k-2}) \subseteq \cdots \subseteq \mcc^2(V;P_0,P_1) \subseteq \mcc^1(V;P_0) = \mcc(V;P_0).
\end{align*}
\end{remark}

\begin{lemma}\label{lem:PolarPointLemma} \textbf{(Polar Point Lemma)}\\
	Let $V \subseteq \PP_K^r$ be an intersection of hypersurfaces and let $\lp P_0,\dotsc,P_k \rp$ be a $k$-polar point of $V$. Then, $L(P_0,\dotsc,P_k) \subseteq V$ is a $k$-plane.
\end{lemma}

\begin{proof}
We prove the claim by induction on $k$ and observe that the case of $k=1$ follows immediately from Lemma \ref{lem:Bertinis Lemma For Intersections Of Hypersurfaces}. Now, consider arbitrary $k > 1$ and let $\lp P_0,\dotsc,P_k \rp$ be a $k$-polar point of $V$. Then, $\lp P_1,\dotsc,P_k \rp$ is a $(k-1)$-polar point of $\mcc(V;P_0)$. Recall that $\mcc(V;P_0)$ is a cone and $L(P_1,\dotsc,P_k) \subseteq \mcc(V;P_0)$ is a $(k-1)$-plane which does not contain the vertex $P_0$, hence $L(P_0,\dotsc,P_k) \subseteq V$ is a $k$-plane.
\end{proof}

\begin{definition}\label{def:The Type Of An Intersection of Hypersurfaces}  \textbf{(Type of an Intersection of Hypersurfaces)}\\
	Given an intersection of hypersurfaces $V \subseteq \PP_K^r$, we say that $V$ is \textbf{of type $\lbm d &d-1 &\cdots &2 &1\\ \ell_d &\ell_{d-1} &\cdots &\ell_2 &\ell_1 \rem$} if 
$V$ has multi-degree
\begin{align*}
	( \underbrace{d,\dotsc,d}_{\ell_d \text{ many}}, \underbrace{d-1,\dotsc,d-1}_{\ell_{d-1} \text{ many}},\cdots,\underbrace{2,\dotsc,2}_{\ell_2 \text{ many}},\underbrace{1,\dotsc,1}_{\ell_1 \text{ many}} ).
\end{align*}

	If $\ell_j=0$ for any $1 \leq j \leq d-1$, we may omit it in the presentation; e.g. an intersection of four quadrics is of type $\lbm 2\\ 4 \rem$. When $d \geq 2$ and each $\ell_j=1$, we abbreviate the notation and say $V$ is of type $(1,\dotsc,d)$. When $d \geq 2$, $\ell_1=0$, and the other $\ell_j=1$, we say $V$ is of type $(2,\dotsc,d)$. 
	
	Consider an intersection of hypersurfaces $V = V_d \cap \cdots \cap V_1$ of type $(1,\dotsc,d)$ in $\PP_K^r$ with $\deg(V_j)=j$. Note that $V_1 \cong \PP_K^{r-1}$ and thus we can also consider $V$ as an intersection of hypersurfaces of type $(2,\dotsc,d)$ inside $V_1 \cong \PP_K^{r-1}$. 
\end{definition}

\begin{prop}\label{prop:Type Of a kth Polar Cone of a type (1,...,d)} \textbf{(Type of a $k^{th}$ Polar Cone of an Intersection of Type $(1,\dotsc,d)$)}\\
	Let $V \subseteq \PP_K^r$ be an intersection of hypersurfaces of type $(1,\dotsc,d)$ and take $k \geq 1$. A $k^{th}$ polar cone $\mcc^k(V;P_0,\dotsc,P_{k-1})$ is of type
\begin{align*}
	\lbm
	d   &d-1     &d-2     &\cdots &3   &2   &1\\
	1 &k+1 &\binom{k+2}{2} &\cdots &\binom{k+d-3}{d-3} &\binom{k+d-2}{d-2} &\binom{k+d-1}{d-1}
	\rem
\end{align*}	
	
\noindent
for $r \geq \binom{k+d}{d-1}$.
\end{prop}

\begin{proof}
	We proceed by induction on $k$ and note that the $k=1$ case follows from the definition of a polar cone of an intersection of hypersurfaces, along with the observation that $\binom{1+j}{j} = 1+j$ for each $2 \leq j \leq n-1$. 
	
	Now, suppose the claim is true for an arbitrary $k$. The number of hypersurfaces of degree $j$ in a $(k+1)^{st}$ polar cone is exactly the number of hypersurfaces of degree at least $j$ in a $k^{th}$ polar cone, e.g. $\SL_{i=0}^j \binom{k+i}{i}$. However,
	\begin{align*}
		\SL_{i=0}^j \binom{k+i}{i} = \binom{k+j+1}{j} = \binom{(k+1)+j}{j},
	\end{align*}
	
	\noindent
	by induction on $j$, using that $\binom{a}{b} = \binom{a-1}{b} + \binom{a-1}{b-1}$ for positive integers $a>b$. This combinatorial argument also yields that $\SL_{i=0}^{d-1} \binom{k+i}{i} = \binom{k+d}{d-1}$.
\end{proof}

In Subsection \ref{subsec:NewBoundsFromPolarCones}, we will use iterated polar cones to establish new upper bounds on $\RD(n)$. We first use iterated polar cones to determine $k$-planes on intersections of quadrics.

\begin{prop}\label{prop:kPolarPointsOnIntersectionsOfQuadrics} \textbf{($k$-Polar Points Intersections of Quadrics)}\\
	Let $V \subseteq \PP_K^r$ be an intersection of hypersurfaces of type $\lbm 2\\ \ell \rem$ and take $k \geq 1$. For $r \geq (k+1)\ell+k$, we can determine a $k$-polar point $\lp P_0,\dotsc,P_k \rp$ over an extension $L/K$ with $\RD(L/K) \leq \RD(2^\ell)$. Moreover, for any point $P \in V(K)$, we can determine a $k$-plane containing $P$ over such an extension $L$.
\end{prop}

\begin{proof} When $r > (k+1)\ell+k$, we can restrict to an arbitrary $(k+1)\ell+k$-plane in $\PP_K^r$ and hence it suffices to consider the case where $r = (k+1)\ell+k$. Note that for a quadric hypersurface, a $k^{th}$ polar cone consists of the original quadric and $k$ hyperplanes $\lp \text{e.g. is of type }\lbm 2 &1\\ 1 &k \rem \rp$. Hence, a $k^{th}$ polar cone of $V$ has type $\lbm 2 &1\\ \ell &k\ell \rem$. We proceed by induction on $k$.
	
	When $k=1$ and $r=2\ell+1$, $\dim(V) \geq \ell+1 > 0$ and Proposition \ref{prop:RationalPointsOverExtensions} allows us to determine a point $P_0 \in V(L_1)$ over an extension $L_1/K$ of degree $2^\ell$. The polar cone $\mcc(V;P_0)$ is of type $\lbm 2 &1\\ \ell &\ell \rem$, hence
\begin{align*}
	\dim\lp \mcc(V;P_0) \rp \geq (2\ell+1)-(2\ell) = 1 > 0.
\end{align*}	
	
\noindent
We use Proposition \ref{prop:RationalPointsOverExtensions} to determine a point $P_1 \in \mcc(V;P_0)(L) \setminus \{P_0\}$ over an extension $L/L_1$ of resolvent degree at most $\RD\lp 2^\ell \rp$. Note that $\lp P_0,P_1 \rp$ is a 1-polar point of $V$ by construction.
	
	We now consider the case of an arbitrary $k > 1$. By induction, we pass to an extension $L_1/K$ of degree at most $2^\ell$ to determine a $(k-1)$-polar point $\lp P_0,\dotsc,P_{k-1} \rp$ of $V$. Observe that
	\begin{align*}
		\dim\lp \mcc^k(V;P_0,\dotsc,P_{k-1}) \rp \geq (k+1)\ell+k  - (k+1)\ell \geq k.
	\end{align*}
	
\noindent
By Proposition \ref{prop:RationalPointsOverExtensions}, we can pass to an extension $L/L_1$ of degree at most $2^\ell$ to determine an $L$-rational point $P_k$ of $\mcc^k(V;P_0,\dotsc,P_{k-1}) \setminus L(P_0,\dotsc,P_{k-1})$. By construction, $\lp P_0,\dotsc,P_k \rp$ is a $k$-polar point of $V$.
	
	For the final claim, note that we can replace $P_0$ with $P$ in the proof and then Lemma \ref{lem:PolarPointLemma} yields $L(P,P_1,\dotsc,P_k)$ is a suitable $k$-plane.
\end{proof}


\section{New Upper Bounds on $\RD(n)$}\label{sec:NewUpperBoundsOnRD(n)}

\subsection{Tschirnhaus Transformations}\label{subsec:Tschirnhaus Transformations}

We recall pertinent information about Tschirnhaus transformations here and refer the reader to \cite{Wolfson2021} for a more complete treatment. Recall that $K_n = \CC(a_1,\dotsc,a_n)$, a purely transcendental extension of $\CC$ with transcendence basis $a_1,\dotsc,a_n$.

\begin{definition}\label{def:GenericPolynomials} \textbf{(General Polynomials)}\\
	The \textbf{general polynomial of degree $n$} is the polynomial
	\begin{align*}
		\phi_n(z) = z^n + a_1z^{n-1} + \cdots + a_{n-1}z + a_n \in K_n[z].
	\end{align*}
\end{definition}

\begin{definition}\label{def:TschirnhausTransformations} \textbf{(Tschirnhaus Transformations)}\\
	A \textbf{Tschirnhaus transformation} of the general degree $n$ polynomial is an isomorphism of $K_n$-fields
	\begin{align*}
		K_n[z]/(\phi_n(z)) \cong K_n[z]/(\psi(z)),
	\end{align*}
	
	\noindent
	where
	\begin{align*}
		\psi(z) = z^n + b_1z^{n-1} + \cdots + b_{n-1}z + b_n.
	\end{align*}
	
	\noindent
	It has \textbf{type $(j_1,\dotsc,j_k)$} if $b_{j_1} = \cdots = b_{j_k}=0$.
\end{definition}

\begin{remark}\label{rem:DescriptionOfTschirnhausTransformations} \textbf{(Description of Tschirnhaus Transformations)}\\
	Given primitive elements $\zeta_\phi$ of $K_n[z]/(\phi_n(z))$ and $\zeta_\psi$ of $K_n[z]/(\psi(z))$, every $K_n$-algebra isomorphism $\Upsilon:K_n[z]/(\phi_n(z)) \ra K_n[z]/(\psi(z))$ is determined by $\Upsilon(\zeta_\phi)$, which is a $K_n$-linear combination of powers of $\zeta_\psi$:
\begin{equation*}
	\Upsilon(\zeta_\phi) = w_0 + w_1\zeta_\psi + \cdots + w_{n-1}\zeta_\psi^{n-1}. 
\end{equation*}

Note that $\Upsilon$ is an isomorphism exactly when there is some $j \geq 1$ such that $w_j \not= 0$. Let $\AAA_{K_n}^n$ be the affine space with coordinates $w_0,\dotsc,w_{n-1}$ and denote the $w_0$-axis of $\AAA_{K_n}^n$ by $\AAA_{K_n,0}^1$. Corollary 3.3 of \cite{Wolfson2021} shows that the space of Tschirnhaus transformations is exactly
\begin{equation*}
	\wt{\mct}_{K_n}^n = \AAA_{K_n}^n \setminus \AAA_{K_n,0}^1.
\end{equation*}

However, we need only work with Tschirnhaus transformations up to re-scaling. Let $\PP_{K_n}^{n-1}$ denote the projective space with coordinates $w_0,\dotsc,w_{n-1}$ and observe that the space of Tschirnhaus transformations up to re-scaling is
\begin{equation*}
	\mct_{K_n}^n = \PP_{K_n}^{n-1} \setminus \lb [1:0:\cdots:0] \rb.
\end{equation*}
\end{remark}
	
	Moreover, each $b_m$ in Definition \ref{def:TschirnhausTransformations} is a homogenous polynomial of degree $m$ in the $w_0,\dotsc,w_{n-1}$ with coefficients in $K_n$ (e.g. an element an of $K_n[w_0,\dotsc,w_{n-1}]_{(m)}$). In the following definition, we build upon the language used in \cite{Wolfson2021}.

\begin{definition}\label{def:TschirnhausCompleteIntersections} \textbf{(Tschirnhaus Complete Intersections)}\\
	Fix $n \in \ZZ_{\geq 1}$. For any $m \in \lb 1,\dotsc,n \rb$, the \textbf{$m^{th}$ extended Tschirnhaus hypersurface} is
\begin{equation*}
	\tau_m := \VV(b_m) \subseteq \PP_{K_n}^n,
\end{equation*}	

\noindent
and the \textbf{$m^{th}$ extended Tschirnhaus complete intersection} is
	\begin{equation*}
		\tau_{1,\dotsc,m} := \tau_1 \cap \cdots \cap \tau_m \subseteq \PP_{K_n}^n.
	\end{equation*}
	
	\noindent
	Similarly, the \textbf{$m^{th}$ Tschirnhaus hypersurface} is
	\begin{equation*}
		\tau_m^\circ := \tau_m \cap \mct_{K_n}^n = \tau_m \setminus \lb [1:0:\cdots:0] \rb, 
	\end{equation*}	
	
	\noindent
	 and the \textbf{$m^{th}$ Tschirnhaus complete intersection} is
	\begin{equation*}
		\tau_{1,\dotsc,m}^\circ := \tau_{1,\dotsc,m} \cap \mct_{K_n}^n \setminus \lb [1:0:\cdots:0] \rb.
	\end{equation*}
\end{definition}

\begin{remark}\label{rem:RDBoundsFromTschirnhausTransformations} \textbf{(RD Bounds from Tschirnhaus Transformations)}\\
Given a Tschirnhaus transformation $\Upsilon$ (up to re-scaling) of type $(1,\dotsc,m-1)$ of $\phi_n(z)$ (i.e. a point of $\tau_{1,\dotsc,m-1}^\circ$), we need only consider the normal form
\begin{equation}\label{eqn:RD Bounds from TT}
	z^n + b_{m}z^{n-m} + \cdots + b_{n-1}z + b_n, 
\end{equation}

\noindent
as $\Upsilon\inv$ takes the roots of (\ref{eqn:RD Bounds from TT}) to the roots of $\phi_n$. We can then re-scale the roots over $\sqrt[n]{b_n}$ (over a cyclic extension) and arrive at the normal form
\begin{equation}\label{eqn:RD Bounds from TT 2}
	z^n + c_{m}z^{n-m} + \cdots + c_{n-1}z + 1.
\end{equation}

\noindent
Note that (\ref{eqn:RD Bounds from TT 2}) is an algebraic function of $n-m$ variables and so if we can determine an $L$-rational point of $\tau_{1,\dotsc,m-1}^\circ$ over an extension $L/K_n$ of sufficiently small resolvent degree, we can conclude that $\RD(n) \leq n-m$.
\end{remark}


\subsection{New Bounds From Iterated Polar Cones}\label{subsec:NewBoundsFromPolarCones}

\begin{definition}\label{def:SubIntersectionsOfGivenLevel} \textbf{(Sub-Intersections of Given Level)}\\
Let $V \subseteq \PP_K^r$ be an intersection of hypersurfaces of type $\lbm d &d-1 &\cdots &2 &1\\ \ell_d &\ell_{d-1} &\cdots &\ell_2 &\ell_1 \rem$ defined by
\begin{align*}
	V = \bigcap\limits_{i=1}^d \bigcap \limits_{j=1}^{\ell_i} \VV(f_{i,j}),
\end{align*}

\noindent
with $\deg(f_{i,j}) = i$ for all $1 \leq j \leq \ell_i$. For each $1 \leq d' \leq d$, the \textbf{sub-intersection of $V$ of level $d'$} is the intersection of all defining hypersurfaces of $V$ of degree $d'$ and is denoted $V_{d'}$, i.e.
\begin{align*}
	V_{d'} := \VV(f_{d',1}) \cap \cdots \cap \VV(f_{d',\ell_{d'}}).
\end{align*}
\end{definition}

\begin{theorem}\label{thm:The n-6 Bound} \textbf{(The $n-6$ Bound)}\\
	For $n \geq 21$, $\RD(n) \leq n-6$.
\end{theorem}

\begin{proof}
	First, suppose that we can determine a plane $\Lambda \subseteq \tau_{1,2,3}^\circ$ over an extension of $K_n$ of low resolvent degree. Observe that $\deg\lp\Lambda \cap \tau_{1,\dotsc,5} \rp = 20$ and, by Proposition \ref{prop:RationalPointsOverExtensions}, we can solve a polynomial of degree at most 20 to determine a point $Q$ of $\Lambda \cap \tau_{1,\dotsc,5}$.
	
	From Lemma \ref{lem:PolarPointLemma}, it suffices to determine a 2-polar point $(P_0,P_1,P_2)$ on $\tau_{1,2,3} \subseteq \mct^n$ such that $L(P_0,P_1,P_2) \subseteq \tau_{1,2,3}^\circ$ for $n=20$ over a suitable extension. Indeed, we show the stronger claim that we can determine such a 2-polar point of $\tau_{1,2,3}$ when $n=19$ over an extension of resolvent degree at most $\RD(12)$, which implies the result for larger $n$.
	
	Recall that when $n=19$, we work in $\PP_{K_n}^{18}$ (as in Remark \ref{rem:DescriptionOfTschirnhausTransformations}). To ensure that the plane $\Lambda$ associated to the 2-polar point we determine lies in $\tau_{1,2,3}^\circ$, we pass to a hyperplane $H$ which does not contain $[1:0:\cdots:0]$. We can determine a point $P_0 \in (\tau_{1,2,3} \cap H)(L_1)$, where $L_1/K_n$ is an extension of degree 6, by Proposition \ref{prop:RationalPointsOverExtensions}. The polar cone $\mcc(\tau_{1,2,3} \cap H;P_0)$ has type $\lbm 3 &2 &1\\ 1 &2 &4 \rem$ and so
	\begin{align*}
		\dim\lp \mcc(\tau_{1,2,3};P_0) \rp \geq 18-7 = 11.
	\end{align*}
	
	\noindent
	Hence, we can determine a point $P_1 \in \mcc(\tau_{1,2,3} \cap H;P_0)(L_2) \setminus \{P_0\}$ over a degree 12 extension $L_2/L_1$.
	
	The second polar cone $\mcc^2(\tau_{1,2,3} \cap H;P_0,P_1)$ has type $\lbm 3 &2 &1\\ 1 &3 &7 \rem$. As $\mcc^2(\tau_{1,2,3} \cap H;P_0,P_1)_1$ is the intersection of 7 hyperplanes, $\mcc^2(\tau_{1,2,3} \cap H;P_0P_1)_1 \cong \PP_{L_2}^{11}$. Denote one of the quadrics defining $\mcc^2(\tau_{1,2,3} \cap H;P_0,P_1)$ by $U$. Applying Proposition \ref{prop:kPolarPointsOnIntersectionsOfQuadrics} to $U \cap \mcc^2(\tau_{1,2,3} \cap H;P_0,P_1)_1$ (with $\ell=1$ and $k=5$) yields that we can determine a 5-plane $\Lambda' \subseteq U \cap \mcc^2(\tau_{1,2,3} \cap H;P_0,P_1)_1$ over a quadratic extension $L_3/L_2$. Note that $\mcc^2(\tau_{1,2,3} \cap H;P_0,P_1) \cap \Lambda'$ has type $\lbm 3 &2\\ 1 &2 \rem$ inside $\Lambda'$ and
	\begin{align*}
		\dim\lp \mcc^2(\tau_{1,2,3} \cap H;P_0,P_1) \cap \Lambda' \rp \geq 5-3 \geq 2. 
	\end{align*}

	Thus, we can determine a point $P_2 \in \lp \mcc^2(\tau_{1,2,3} \cap H;P_0,P_1) \cap \Lambda' \rp \setminus L(P_0,P_1)$ over an extension $L_4/L_3$ of degree
	\begin{align*}
		\deg\lp \mcc^2(\tau_{1,2,3} \cap H;P_0,P_1) \cap \Lambda' \rp = 3 \cdot 2^2 = 12,
	\end{align*}
	
	\noindent
	and $\lp P_0, P_1, P_2 \rp$ is a 2-polar point of $\tau_{1,2,3} \cap H$ by construction.
\end{proof}

\begin{remark}\label{rem:Fixing Gaps in Chebotarev} \textbf{(Fixing Gaps in \cite{Chebotarev1954})}\\
	As noted in Section \ref{sec:Introduction}, Chebotarev gave an argument that $\RD(n) \leq n-6$ for $n \geq 21$ in \cite{Chebotarev1954}. However, his argument had gaps (similar to those of Wiman in \cite{Wiman1927} and which are also in the translation \cite{Sutherland2021A}). There are two specific issues. First, Chebotarev uses a geometric argument of Wiman which assumed without proof that certain intersections of hypersurfaces in affine space were generic (see p.190-191 of \cite{Chebotarev1954} or p.3 of the translation \cite{Sutherland2021B}). Additionally, Lemma 1 and Lemma 2 of \cite{Chebotarev1954} (Lemma 2.1 and Lemma 2.2 of the translation \cite{Sutherland2021B}) also assume, but do not prove, that certain intersections of affine hypersurfaces are generic. The above proof of Theorem \ref{thm:The n-6 Bound} fixes the issues associated to the use of Lemma 1 and Lemma 2 in the proof of the un-numbered theorem of \cite{Chebotarev1954} (Theorem 2.3 of \cite{Sutherland2021B}). Additionally, this argument can be adapted to give a geometric proof of Wiman's claims, although Dixmier has already given an algebraic proof in the appendix of \cite{Dixmier1993}.
\end{remark}

\begin{remark}\label{rem:NotationForPolarPointsOfA1234} \textbf{(Notation for Theorem \ref{thm:The n-7,...,n-14 Bounds})}\\
	We move from finding a plane on $\tau_{1,2,3}^\circ$ to finding $k$-planes on $\tau_{1,2,3,4}^\circ$ to improve the bound on $n$ for which $\RD(n) \leq n-m$ for $m$ between $7$ and $14$. To simplify the statement of Theorem \ref{thm:The n-7,...,n-14 Bounds}, we introduce two functions $\rho:\{1,2,3,4,5,6,7,8,9\} \ra \ZZ \text{ and } \eta:\{1,2,3,4,5,6,7,8,9\} \ra \ZZ$. For each $k$, $\rho(k)$ is the ambient dimension needed in our method to determine a $k$-plane on $\tau_{1,2,3,4}$ and $\eta(k)$ is the degree of the largest extension needed. The functions $\rho$ and $\eta$ are defined by the following values:
	\begin{center}
	\begin{tabular}{|c|c|c|c|c|c|c|c|c|c|}
	\hline
	$k$ &1 &2 &3 &4 &5 &6 &7 &8 &9\\
	\hline
	$\rho(k)$ &25 &60 &264 &806 &1773 &8905 &34546 &77040 &612581\\
	\hline
	$\eta(k)$ &36 &108 &324 &972 &2916 &8748 &26244 &78732 &236196\\
	\hline
	\end{tabular}
	\end{center}
\end{remark}

\begin{theorem}\label{thm:The n-7,...,n-14 Bounds} \textbf{(The $n-7,\dotsc,n-14$ Bounds)}
	\begin{enumerate}
		\item For $n \geq 109$, $\RD(n) \leq n-7$.
		\item For $n \geq 325$, $\RD(n) \leq n-8$.
		\item For each $9 \leq \ell \leq 14$ and $n > \frac{(\ell-1)!}{24}$, $\RD(n) \leq n-\ell$.
	\end{enumerate}
\end{theorem}

\begin{proof} We claim that we can determine a $k$-plane $\Lambda \subseteq \tau_{1,2,3,4}^\circ$ over an extension of $K_n$ of resolvent degree at most $\RD(\eta(k))$ when $n \geq \rho(k)$. Given this claim, $\Lambda \cap \tau_{1,\dotsc,k+4}$ has degree $\frac{(k+4)!}{24}$. By Proposition \ref{prop:RationalPointsOverExtensions}, we can solve a polynomial of degree at most $\frac{(k+4)!}{24}$ to determine a point $Q$ of $\Lambda \cap \tau_{1,\dotsc,k+4}^\circ$. Then, the new bounds on resolvent degree follow immediately from Remark \ref{rem:RDBoundsFromTschirnhausTransformations}. Note that for the $n-7$ bound (respectively, $n-8$ bound), $k=2$ (respectively, $k=3$) and $\eta(k) \geq \frac{(k+4)!}{24}$.
	
	It remains to determine the $k$-planes on $\tau_{1,2,3,4}^\circ$ and, by Lemma \ref{lem:PolarPointLemma}, it suffices to construct a $k$-polar point $(P_0,\dotsc,P_k)$ of $\tau_{1,2,3,4}$ which does not contain $[1:0:\cdots:0]$. For each $k$, it suffices to prove the claim when $n=\rho(k)$, as we can always restrict to a $\PP_{K_n}^{\rho(k)}$ when $n > \rho(k)$. In each case, we immediately pass to a hypersurface which does not contain $[1:0:\cdots:0]$ and, to simplify notation, the computations that follow all start within this hypersurface. Recall that the space of Tschirnhaus transformations (up to re-scaling) is $\PP_{K_n}^{n-1}$ and we are passing to a hypersurface, so for each $k$ we work in a $\PP_{K_n}^{\rho(k)-2}$.
	In each case, the extensions will be enumerated as $L_j$ and we always begin with $L_1/K_n$. Similarly, the $k$-planes we determine will be enumerated as $\Lambda_\ell$ and we always begin with $\Lambda_1$. Additionally, Proposition \ref{prop:Type Of a kth Polar Cone of a type (1,...,d)} gives the type of the relevant polar cone.
	
\paragraph{Case: $k=1$.}
Recall that $\rho(1)=25$ and thus we work in $\PP_{K_n}^{23}$. By Proposition \ref{prop:RationalPointsOverExtensions}, we can determine a point $P_0$ of $\tau_{1,2,3,4}$ over an extension $L_1/K_n$ of degree at most 24. The polar cone $\mcc(\tau_{1,2,3,4};P_0)$ has type $\lbm 4 &3 &2 &1\\ 1 &2 &3 &4 \rem$. Note that $\mcc(\tau_{1,2,3,4};P_0)_1$ is an intersection of 4 hyperplanes and thus $\mcc(\tau_{1,2,3,4};P_0)_1 \cong \PP_{L_1}^{19}$. Since $\mcc(\tau_{1,2,3,4};P_0)_2 \cap \mcc(\tau_{1,2,3,4};P_0)_1$ has type $\lbm 2\\ 3\\ \rem$ inside $\mcc(\tau_{1,2,3,4};P_0)_1$ and
	\begin{equation*}
		19 = (4+1)(3)+4,
	\end{equation*}

\noindent
Proposition \ref{prop:kPolarPointsOnIntersectionsOfQuadrics} (with $k=4$ and $\ell=3$) allows us to determine a 4-plane $\Lambda_1 \subseteq \mcc(\tau_{1,2,3,4};P_0)_2 \cap \mcc(\tau_{1,2,3,4};P_0)_1$ over an extension $L_2/L_1$ of degree at most 8. Observe that $\mcc(\tau_{1,2,3,4};P_0) \cap \Lambda_1$ has type $\lbm 4 &1\\ 1 &2 \rem$, hence
\begin{equation*}
	\dim\lp \mcc(\tau_{1,2,3,4};P_0) \cap \Lambda_1 \rp \geq 4-3 \geq 1.
\end{equation*}

Consequently, we can determine a rational point $P_1$ of $\mcc(\tau_{1,2,3,4};P_0) \setminus \{P_0\}$ over an extension $L_3/L_2$ of degree at most 36. By construction, $(P_0, P_1)$ is a 1-polar point of $\tau_{1,2,3,4}$.
	
\paragraph{Case: $k=2$.}	
	Note that $\rho(2)=60$, so we work in $\PP_{K_n}^{58}$. From the $k=1$ case, we pass to an extension $L_1/K_n$ of resolvent degree at most $\RD(36)$ and assume that we have a 1-polar point $(P_0, P_1)$. The second polar cone $\mcc^2(\tau_{1,2,3,4};P_0,P_1)$ has type  $\lbm 4 &3 &2 &1\\ 1 &3 &6 &10 \rem$. Observe that $\mcc^2(\tau_{1,2,3,4};P_0,P_1)_1$ is an intersection of 10 hyperplanes, hence $\mcc^2(\tau_{1,2,3,4};P_0,P_1)_1 \cong \PP_{L_1}^{48}$. Moreover, $\mcc^2(\tau_{1,2,3,4};P_0,P_1)_2 \cap \mcc^2(\tau_{1,2,3,4};P_0,P_1)_1$ has type $\lbm 2\\ 6 \rem$ in $\mcc^2(\tau_{1,2,3,4};P_0,P_1)_1$ and
\begin{equation*}
	48 = (5+1)(6)+5.
\end{equation*}

Thus, applying Proposition \ref{prop:kPolarPointsOnIntersectionsOfQuadrics} (with $k=5$ and $\ell=6$) allows us to determine a 5-plane $\Lambda_1 \subseteq \mcc^2(\tau_{1,2,3,4};P_0,P_1)_2 \cap \mcc^2(\tau_{1,2,3,4};P_0,P_1)_1$ over an extension $L_2/L_1$ of degree at most 64. Note that $\mcc^2(\tau_{1,2,3,4};P_0,P_1) \cap \Lambda_1$ has type $\lbm 4 &3\\ 1 &3 \rem$ inside $\Lambda_1$, hence
\begin{equation*}
	\dim\lp \mcc^2(\tau_{1,2,3,4};P_0,P_1) \cap \Lambda_1 \rp \geq 6-4 \geq 2. 
\end{equation*}

As a result, we can determine a rational point $P_2$ of $\mcc^2(\tau_{1,2,3,4};P_0,P_1) \setminus L(P_0,P_1)$ over an extension $L_3/L_2$ of degree at most 108 and our method ensures that $(P_0,P_1,P_2)$ is a 2-polar point of $\tau_{1,2,3,4}$.

\paragraph{Case: $k=3$.}
Recall that $\rho(3)=264$ and so we work in $\PP_{K_n}^{262}$. Using the $k=2$ case, we pass to an extension $L_1/K_n$ of resolvent degree at most $\RD(108)$ and assume that we have a 2-polar point $(P_0,P_1,P_2)$. Observe that the third polar cone $\mcc^3(\tau_{1,2,3,4};P_0,P_1,P_2)$ has type $\lbm 4 &3 &2 &1\\ 1 &4 &10 &20 \rem$. Thus, $\mcc^3(\tau_{1,2,3,4};P_0,P_1,P_2)_1$ is an intersection of 35 hyperplanes and $\mcc^3(\tau_{1,2,3,4};P_0,P_1,P_2)_1 \cong \PP_{L_1}^{242}$. Inside $\mcc^3(\tau_{1,2,3,4};P_0,P_1,P_2)_1$, $\mcc^3(\tau_{1,2,3,4};P_0,P_1,P_2)_2 \cap \mcc^3(\tau_{1,2,3,4};P_0,P_1,P_2)_1$ has type $\lbm 2\\ 10 \rem$. However, $2^8 = 256$ is the largest power of 2 less than $4 \cdot 3^4 = 324$. As a result, we split $\mcc^3(\tau_{1,2,3,4};P_0,P_1,P_2)_2$ into two intersections of quadrics:
\begin{equation*}
	\mcc^3(\tau_{1,2,3,4};P_0,P_1,P_2)_2 = W_1 \cap W_2,
\end{equation*}

\noindent
where $W_1$ has type $\lbm 2\\ 2 \rem$ and $W_2$ has type $\lbm 2\\ 8 \rem$. Next, observe that
\begin{equation*}
	242 = (80+1)(2)+80,
\end{equation*}

\noindent
and thus applying Proposition \ref{prop:kPolarPointsOnIntersectionsOfQuadrics} (with $k=75$ and $\ell=2$) allows us to determine an 80-plane $\Lambda_1 \subseteq W_1 \cap \mcc^3(\tau_{1,2,3,4};P_0,P_1,P_2)_1$ over an extension $L_2/L_1$ of degree at most 4. Also, $W_2 \cap \Lambda_1$ has type $\lbm 2\\ 8 \rem$ in $\Lambda_1$ and
\begin{equation*}
	80 = (8+1)(8)+8.
\end{equation*}

We then apply Proposition \ref{prop:kPolarPointsOnIntersectionsOfQuadrics} (with $k=8$ and $\ell=8$) to determine an 8-plane $\Lambda_2 \subseteq W_2 \cap \Lambda_1$ over an extension $L_3/L_2$ of degree at most 256. Now, $\mcc^3(\tau_{1,2,3,4};P_0,P_1,P_2) \cap \Lambda_2$ has type $\lbm 4 &3\\ 1 &4 \rem$ in $\Lambda_2$. Consequently,
\begin{equation*}
	\dim\lp \mcc^3(\tau_{1,2,3,4};P_0,P_1,P_2) \cap \Lambda_2 \rp \geq 8-5 \geq 3
\end{equation*}

\noindent
and we can determine a rational point $P_3$ of $\mcc^3(\tau_{1,2,3,4};P_0,P_1,P_2) \setminus L(P_0,P_1,P_2)$ over an extension $L_4/L_3$ of degree at most 324. By design, $(P_0,P_1,P_2,P_3)$ is a 3-polar point of $\tau_{1,2,3,4}$.

\paragraph{Case: $k=4$.}
Observe that $\rho(4)=806$, hence we work in $\PP_{K_n}^{804}$. By the $k=3$ case, we pass to an extension $L_1/K_n$ of resolvent degree at most $\RD(324)$ and assume that we have a 3-polar point $(P_0,P_1,P_2,P_3)$. Observe that the fourth polar cone $\mcc^4(\tau_{1,2,3,4};P_0,\dotsc,P_3)$ has type $\lbm 4 &3 &2 &1\\ 1 &5 &15 &35 \rem$. Thus, $\mcc^4(\tau_{1,2,3,4};P_0,\dotsc,P_3)_1$ is an intersection of 35 hyperplanes and $\mcc^4(\tau_{1,2,3,4};P_0,\dotsc,P_3)_1 \cong \PP_{L_1}^{769}$. Note that $\mcc^4(\tau_{1,2,3,4};P_0,\dotsc,P_3)_2 \cap \mcc^4(\tau_{1,2,3,4};P_0,\dotsc,P_3)_1$ has type $\lbm 2\\ 15 \rem$ in $\mcc^4(\tau_{1,2,3,4};P_0,\dotsc,P_3)_1$ and $2^9 = 512$ is the largest power of 2 less than $4 \cdot 3^5 = 972$. Consequently, we split $\mcc^4(\tau_{1,2,3,4};P_0,\dotsc,P_3)_2$ into two intersections of quadrics:
\begin{equation*}
	\mcc^4(\tau_{1,2,3,4};P_0,\dotsc,P_3)_2 = W_1 \cap W_2,
\end{equation*}

\noindent
where $W_1$ has type $\lbm 2\\ 6 \rem$ and $W_2$ has type $\lbm 2\\ 9 \rem$. We have that
\begin{equation*}
	769 = (109+1)(6)+109,
\end{equation*}

\noindent
and so by applying Proposition \ref{prop:kPolarPointsOnIntersectionsOfQuadrics} (with $k=109$ and $\ell=6$), we can determine a 109-plane $\Lambda_1 \subseteq W_1 \cap \mcc^4(\tau_{1,2,3,4};P_0,\dotsc,P_3)_1$ over an extension $L_2/L_1$ of degree at most 64. Similarly, $W_2 \cap \Lambda_1$ has type $\lbm 2\\ 9 \rem$ in $\Lambda_1$ and
\begin{equation*}
	109 = (10+1)(9)+10.
\end{equation*}

By applying Proposition \ref{prop:kPolarPointsOnIntersectionsOfQuadrics} (with $k=10$ and $\ell=9$), we determine an 8-plane $\Lambda_2 \subseteq W_2 \cap \Lambda_1$ over an extension $L_3/L_2$ of degree at most 512. It follows that $\mcc^4(\tau_{1,2,3,4};P_0,\dotsc,P_3) \cap \Lambda_2$ has type $\lbm 4 &3\\ 1 &5 \rem$ in $\Lambda_2$. Consequently,
\begin{equation*}
	\dim\lp \mcc^4(\tau_{1,2,3,4};P_0,\dotsc,P_3) \cap \Lambda_2 \rp \geq 10-6 \geq 4
\end{equation*}

\noindent
and we can determine a rational point $P_4$ of $\mcc^4(\tau_{1,2,3,4};P_0,\dotsc,P_3) \setminus L(P_0,P_1,P_2,P_3)$ over an extension $L_4/L_3$ of degree at most 972. Indeed, $(P_0,\dotsc,P_4)$ is a 4-polar point of $\tau_{1,2,3,4}$.

\paragraph{Case: $k=5$.}
Observe that $\rho(5)=1773$, hence we work in $\PP_{K_n}^{1771}$. From the $k=4$ case, we pass to an extension $L_1/K_n$ of resolvent degree at most $\RD(972)$ and assume that we have a 4-polar point $(P_0,\dotsc,P_4)$. The fifth polar cone $\mcc^5(\tau_{1,2,3,4};P_0,\dotsc,P_4)$ has type $\lbm 4 &3 &2 &1\\ 1 &6 &21 &56 \rem$. Hence, $\mcc^5(\tau_{1,2,3,4};P_0,\dotsc,P_4)_1$ is an intersection of 56 hyperplanes and $\mcc^5(\tau_{1,2,3,4};P_0,\dotsc,P_4)_1 \cong \PP_{L_1}^{1715}$. Note that $\mcc^5(\tau_{1,2,3,4};P_0,\dotsc,P_4)_2 \cap \mcc^5(\tau_{1,2,3,4};P_0,\dotsc,P_4)_1$ has type $\lbm 2\\ 21 \rem$ in $\mcc^5(\tau_{1,2,3,4};P_0,\dotsc,P_4)_1$. Observe that $2^{11} = 2048$ is the largest power of 2 less than $4 \cdot 3^6 = 2916$ and so we split $\mcc^5(\tau_{1,2,3,4};P_0,\dotsc,P_4)_2$ into two intersections of quadrics:
\begin{equation*}
	\mcc^5(\tau_{1,2,3,4};P_0,\dotsc,P_4)_2 = W_1 \cap W_2,
\end{equation*}

\noindent
where $W_1$ has type $\lbm 2\\ 10 \rem$ and $W_2$ has type $\lbm 2\\ 11 \rem$. Note that
\begin{equation*}
	1715 = (155+1)(10)+155,
\end{equation*}

\noindent
and, by applying Proposition \ref{prop:kPolarPointsOnIntersectionsOfQuadrics} (with $k=155$ and $\ell=10$), we can determine a 155-plane $\Lambda_1 \subseteq W_1 \cap \mcc^5(\tau_{1,2,3,4};P_0,\dotsc,P_4)_1$ over an extension $L_2/L_1$ of degree at most 1024. Also, $W_2 \cap \Lambda_1$ has type $\lbm 2\\ 11 \rem$ in $\Lambda_1$ and
\begin{equation*}
	155 = (12+1)(11)+12,
\end{equation*}

\noindent
so we can apply Proposition \ref{prop:kPolarPointsOnIntersectionsOfQuadrics} (with $k=12$ and $\ell=11$) to determine a 10-plane $\Lambda_2 \subseteq W_2 \cap \Lambda_1$ over an extension $L_3/L_2$ of degree at most 2048. It follows that $\mcc^5(\tau_{1,2,3,4};P_0,\dotsc,P_4) \cap \Lambda_2$ has type $\lbm 4 &3\\ 1 &6 \rem$ in $\Lambda_2$. Consequently,
\begin{equation*}
	\dim\lp \mcc^5(\tau_{1,2,3,4};P_0,\dotsc,P_4) \cap \Lambda_2 \rp \geq 12-7 \geq 5,
\end{equation*}

\noindent
and we can determine a rational point $P_5$ of $\mcc^5(\tau_{1,2,3,4};P_0,\dotsc,P_4) \setminus L(P_0,\dotsc,P_4)$ over an extension $L_4/L_3$ of degree at most 2916. By construction, $(P_0,\dotsc,P_5)$ is a 5-polar point of $\tau_{1,2,3,4}$.

\paragraph{Case: $k=6$.}
Recall that $\rho(6)=8905$ and thus we work in $\PP_{K_n}^{8903}$. Using the $k=5$ case, we pass to an extension $L_1/K_n$ of resolvent degree at most $\RD(2916)$ and assume that we have a 5-polar point $(P_0,\dotsc,P_5)$. The sixth polar cone $\mcc^6(\tau_{1,2,3,4};P_0,\dotsc,P_5)$ has type $\lbm 4 &3 &2 &1\\ 1 &7 &28 &84 \rem$. It follows that $\mcc^6(\tau_{1,2,3,4};P_0,\dotsc,P_5)_1$ is an intersection of 84 hyperplanes, so $\mcc^6(\tau_{1,2,3,4};P_0,\dotsc,P_5)_1 \cong \PP_{L_1}^{8819}$. Note that $\mcc^6(\tau_{1,2,3,4};P_0,\dotsc,P_5)_2 \cap \mcc^6(\tau_{1,2,3,4};P_0,\dotsc,P_5)_1$ has type $\lbm 2\\ 28 \rem$ in $\mcc^6(\tau_{1,2,3,4};P_0,\dotsc,P_5)_1$ and  $2^{13} = 8192$ is the largest power of 2 less than $4 \cdot 3^7 = 8748$. Consequently, we split $\mcc^6(\tau_{1,2,3,4};P_0,\dotsc,P_5)_2$ into three intersections of quadrics:
\begin{equation*}
	\mcc^6(\tau_{1,2,3,4};P_0,\dotsc,P_5)_2 = W_1 \cap W_2 \cap W_3,
\end{equation*}

\noindent
where $W_1$ has type $\lbm 2\\ 2 \rem$ and both $W_2,W_3$ have type $\lbm 2\\ 13 \rem$. We have that
\begin{equation*}
	8819 = (2939+1)(2)+2939,
\end{equation*}

\noindent
and so by applying Proposition \ref{prop:kPolarPointsOnIntersectionsOfQuadrics} (with $k=2939$ and $\ell=2$), we can determine a 2939-plane $\Lambda_1 \subseteq W_1 \cap \mcc^6(\tau_{1,2,3,4};P_0,\dotsc,P_5)_1$ over an extension $L_2/L_1$ of degree at most 4. Additionally, $W_2 \cap \Lambda_1$ has type $\lbm 2\\ 13 \rem$ inside $\Lambda_1$ and
\begin{equation*}
	2939 = (209+1)(13)+209.
\end{equation*}

We then apply Proposition \ref{prop:kPolarPointsOnIntersectionsOfQuadrics} (with $k=209$ and $\ell=13$) to determine a 209-plane $\Lambda_2 \subseteq W_2 \cap \Lambda_1$ over an extension $L_3/L_2$ of degree at most 8192. Similarly, $W_3 \cap \Lambda_2$ has type $\lbm 2\\ 13 \rem$ in $\Lambda_2$ and
\begin{equation*}
	209 = (14+1)(13)+14,
\end{equation*}

\noindent
and thus we apply Proposition \ref{prop:kPolarPointsOnIntersectionsOfQuadrics} (with $k=14$ and $\ell=13$) to determine a 14-plane $\Lambda_3 \subseteq W_3 \cap \Lambda_2$ over an extension $L_4/L_3$ of degree at most 8192. Observe that $\mcc^6(\tau_{1,2,3,4};P_0,\dotsc,P_5) \cap \Lambda_3$ has type $\lbm 4 &3\\ 1 &7 \rem$ inside $\Lambda_3$ and
\begin{equation*}
	\dim\lp \mcc^6(\tau_{1,2,3,4};P_0,\dotsc,P_5) \cap \Lambda_3 \rp \geq 14-8 \geq 6.
\end{equation*}

As a consequence, we can determine a rational point $P_6$ of $\mcc^6(\tau_{1,2,3,4};P_0,\dotsc,P_5) \setminus L(P_0,\dotsc,P_5)$ over an extension $L_5/L_4$ of degree at most 8748. It follows that $(P_0,\dotsc,P_6)$ is a 6-polar point of $\tau_{1,2,3,4}$.

\paragraph{Case: $k=7$.}
Note that $\rho(7)=34546$, so we work in $\PP_{K_n}^{34544}$. By the $k=6$ case, we pass to an extension $L_1/K_n$ of resolvent degree at most $\RD(8748)$ and assume that we have a 6-polar point $(P_0,\dotsc,P_6)$. Observe that the seventh polar cone $\mcc^7(\tau_{1,2,3,4};P_0,\dotsc,P_6)$ has type $\lbm 4 &3 &2 &1\\ 1 &8 &36 &120 \rem$ and so $\mcc^7(\tau_{1,2,3,4};P_0,\dotsc,P_6)_1$ is an intersection of 120 hyperplanes. Thus, $\mcc^7(\tau_{1,2,3,4};P_0,\dotsc,P_6)_1 \cong \PP_{L_1}^{34424}$. Inside $\mcc^7(\tau_{1,2,3,4};P_0,\dotsc,P_6)_1$, $\mcc^7(\tau_{1,2,3,4};P_0,\dotsc,P_6)_2 \cap \mcc^7(\tau_{1,2,3,4};P_0,\dotsc,P_6)_1$ has type $\lbm 2\\ 36 \rem$. Note that $2^{14} = 16384$ is the largest power of 2 less than $4 \cdot 3^8 = 26244$. Consequently, we split $\mcc^7(\tau_{1,2,3,4};P_0,\dotsc,P_6)_2$ into three intersections of quadrics:
\begin{equation*}
	\mcc^7(\tau_{1,2,3,4};P_0,\dotsc,P_6)_2 = W_1 \cap W_2 \cap W_3,
\end{equation*}

\noindent
where $W_1$ has type $\lbm 2\\ 8 \rem$ and both $W_2,W_3$ have type $\lbm 2\\ 14 \rem$. We have that
\begin{equation*}
	34424 = (3824+1)(8)+3824,
\end{equation*}

\noindent
and apply Proposition \ref{prop:kPolarPointsOnIntersectionsOfQuadrics} (with $k=3824$ and $\ell=8$) to determine a 3824-plane $\Lambda_1 \subseteq W_1 \cap \mcc^7(\tau_{1,2,3,4};P_0,\dotsc,P_6)_1$ over an extension $L_2/L_1$ of degree at most 256. Additionally, $W_2 \cap \Lambda_1$ has type $\lbm 2\\ 14 \rem$ inside $\Lambda_1$ and
\begin{equation*}
	3824 = (254+1)(14)+254.
\end{equation*}

Consequently, we can apply Proposition \ref{prop:kPolarPointsOnIntersectionsOfQuadrics} (with $k=254$ and $\ell=14$) to determine a 254-plane $\Lambda_2 \subseteq W_2 \cap \Lambda_1$ over an extension $L_3/L_2$ of degree at most 16384. Observe that $W_3 \cap \Lambda_2$ has type $\lbm 2\\ 14 \rem$ in $\Lambda_2$ and
\begin{equation*}
	254 = (16+1)(14)+16,
\end{equation*}

\noindent
and so, by Proposition \ref{prop:kPolarPointsOnIntersectionsOfQuadrics} (with $k=16$ and $\ell=14$) to determine a 16-plane $\Lambda_3 \subseteq W_3 \cap \Lambda_2$ over an extension $L_4/L_3$ of degree at most 16384. Observe that $\mcc^7(\tau_{1,2,3,4};P_0,\dotsc,P_6) \cap \Lambda_3$ has type $\lbm 4 &3\\ 1 &8 \rem$ inside $\Lambda_3$ and
\begin{equation*}
	\dim\lp \mcc^7(\tau_{1,2,3,4};P_0,\dotsc,P_6) \cap \Lambda_3 \rp \geq 16-9 \geq 7.
\end{equation*}

It follows that we can determine a rational point $P_7$ of $\mcc^7(\tau_{1,2,3,4};P_0,\dotsc,P_6) \setminus L(P_0,\dotsc,P_6)$ over an extension $L_5/L_4$ of degree at most 26244. Note that $(P_0,\dotsc,P_7)$ is a 7-polar point of $\tau_{1,2,3,4}$.

\paragraph{Case: $k=8$.}
Recall that $\rho(8)=77040$ and thus we work in $\PP_{K_n}^{77038}$. By the $k=7$ case, we pass to an extension $L_1/K_n$ of resolvent degree at most $\RD(26244)$ and assume that we have a 7-polar point $(P_0,\dotsc,P_7)$. Observe that the eighth polar cone $\mcc^8(\tau_{1,2,3,4};P_0,\dotsc,P_7)$ has type $\lbm 4 &3 &2 &1\\ 1 &9 &45 &165 \rem$ and so $\mcc^8(\tau_{1,2,3,4};P_0,\dotsc,P_7)_1$ is an intersection of 165 hyperplanes. Thus, $\mcc^8(\tau_{1,2,3,4};P_0,\dotsc,P_7)_1 \cong \PP_{L_1}^{76873}$ and $\mcc^8(\tau_{1,2,3,4};P_0,\dotsc,P_7)_2 \cap \mcc^8(\tau_{1,2,3,4};P_0,\dotsc,P_7)_1$ has type $\lbm 2\\ 45 \rem$ inside $\mcc^8(\tau_{1,2,3,4};P_0,\dotsc,P_7)_1$. We have that $2^{16} = 65536$ is the largest power of 2 less than $4 \cdot 3^9 = 78732$. Consequently, we split $\mcc^8(\tau_{1,2,3,4};P_0,\dotsc,P_7)_2$ into three intersections of quadrics:
\begin{equation*}
	\mcc^8(\tau_{1,2,3,4};P_0,\dotsc,P_7)_2 = W_1 \cap W_2 \cap W_3,
\end{equation*}

\noindent
where $W_1$ has type $\lbm 2\\ 13 \rem$ and both $W_2,W_3$ have type $\lbm 2\\ 16 \rem$. Observe that
\begin{equation*}
	76873 = (5490+1)(13)+5490.
\end{equation*}

\noindent
Thus, we can apply Proposition \ref{prop:kPolarPointsOnIntersectionsOfQuadrics} (with $k=5490$ and $\ell=13$) to determine a 5490-plane $\Lambda_1 \subseteq W_1 \cap \mcc^8(\tau_{1,2,3,4};P_0,\dotsc,P_7)_1$ over an extension $L_2/L_1$ of degree at most 8192. We also have that $W_2 \cap \Lambda_1$ has type $\lbm 2\\ 16 \rem$ inside $\Lambda_1$ and
\begin{equation*}
	5490 = (322+1)(16)+322.
\end{equation*}

As a result, we apply Proposition \ref{prop:kPolarPointsOnIntersectionsOfQuadrics} (with $k=322$ and $\ell=16$) to determine a 322-plane $\Lambda_2 \subseteq W_2 \cap \Lambda_1$ over an extension $L_3/L_2$ of degree at most 65536. Note that $W_3 \cap \Lambda_2$ has type $\lbm 2\\ 16 \rem$ in $\Lambda_2$ and
\begin{equation*}
	322 = (18+1)(16)+18,
\end{equation*}

\noindent
and so Proposition \ref{prop:kPolarPointsOnIntersectionsOfQuadrics} (with $k=18$ and $\ell=16$) allows us to determine an 18-plane $\Lambda_3 \subseteq W_3 \cap \Lambda_2$ over an extension $L_4/L_3$ of degree at most 65536. Observe that $\mcc^8(\tau_{1,2,3,4};P_0,\dotsc,P_7) \cap \Lambda_3$ has type $\lbm 4 &3\\ 1 &9 \rem$ inside $\Lambda_3$ and
\begin{equation*}
	\dim\lp \mcc^8(\tau_{1,2,3,4};P_0,\dotsc,P_7) \cap \Lambda_3 \rp \geq 18-10 \geq 8.
\end{equation*}

Consequently, we can determine a rational point $P_8$ of $\mcc^8(\tau_{1,2,3,4};P_0,\dotsc,P_7) \setminus L(P_0,\dotsc,P_7)$ over an extension $L_5/L_4$ of degree at most 78732. By design, $(P_0,\dotsc,P_8)$ is an 8-polar point of $\tau_{1,2,3,4}$.

\paragraph{Case: $k=9$.}
Recall that $\rho(9)=612581$ and thus we work in $\PP_{K_n}^{612579}$. By the $k=8$ case, we pass to an extension $L_1/K_n$ of resolvent degree at most $\RD(78732)$ and assume that we have a 8-polar point $(P_0,\dotsc,P_8)$. Note that the ninth polar cone $\mcc^9(\tau_{1,2,3,4};P_0,\dotsc,P_8)$ has type $\lbm 4 &3 &2 &1\\ 1 &10 &55 &220 \rem$. So, $\mcc^9(\tau_{1,2,3,4};P_0,\dotsc,P_8))_1$ is an intersection of 220 hyperplanes and $\mcc^9(\tau_{1,2,3,4};P_0,\dotsc,P_8)_1 \cong \PP_{L_1}^{612359}$.  Observe that $\mcc^9(\tau_{1,2,3,4};P_0,\dotsc,P_8)_2 \cap \mcc^9(\tau_{1,2,3,4};P_0,\dotsc,P_8)_1$ has type $\lbm 2\\ 55 \rem$ inside $\mcc^9(\tau_{1,2,3,4};P_0,\dotsc,P_8)_1$ and $2^{17} = 131072$ is the largest power of 2 less than $4 \cdot 3^{10} = 236196$. Correspondingly, we split $\mcc^9(\tau_{1,2,3,4};P_0,\dotsc,P_8)_2$ into four intersections of quadrics:
\begin{equation*}
	\mcc^9(\tau_{1,2,3,4};P_0,\dotsc,P_8)_2 = W_1 \cap W_2 \cap W_3 \cap W_4,
\end{equation*}

\noindent
where $W_1$ has type $\lbm 2\\ 4 \rem$ and all of $W_2,W_3,W_4$ have type $\lbm 2\\ 17 \rem$. Note that
\begin{equation*}
	612359 = (122471+1)(4)+122471,
\end{equation*}

\noindent
and so, by Proposition \ref{prop:kPolarPointsOnIntersectionsOfQuadrics} (with $k=122471$ and $\ell=4$), we can determine a 122471-plane $\Lambda_1 \subseteq W_1 \cap \mcc^9(\tau_{1,2,3,4};P_0,\dotsc,P_8)_1$ over an extension $L_2/L_1$ of degree at most 16. Additionally, $W_2 \cap \Lambda_1$ has type $\lbm 2\\ 17 \rem$ inside $\Lambda_1$ and
\begin{equation*}
	122471 = (6803+1)(17)+6805.
\end{equation*}

As a result, we apply Proposition \ref{prop:kPolarPointsOnIntersectionsOfQuadrics} (with $k=6805$ and $\ell=17$) to determine a 6805-plane $\Lambda_2 \subseteq W_2 \cap \Lambda_1$ over an extension $L_3/L_2$ of degree at most 131072. Note that $W_3 \cap \Lambda_2$ has type $\lbm 2\\ 17 \rem$ in $\Lambda_2$ and
\begin{equation*}
	6803 = (377+1)(17)+377.
\end{equation*}

\noindent
and so Proposition \ref{prop:kPolarPointsOnIntersectionsOfQuadrics} (with $k=18$ and $\ell=16$) allows us to determine an 18-plane $\Lambda_3 \subseteq W_3 \cap \Lambda_2$ over an extension $L_4/L_3$ of degree at most 65536. Also, $W_4 \cap \Lambda_3$ has type $\lbm 2\\ 17 \rem$ in $\Lambda_3$ and
\begin{equation*}
	377 = (20+1)(17)+20.
\end{equation*}

From Proposition \ref{prop:kPolarPointsOnIntersectionsOfQuadrics} (with $k=18$ and $\ell=16$), we can determine a 20-plane $\Lambda_4 \subseteq W_4 \cap \Lambda_3$ over an extension $L_5/L_4$ of degree at most 131072. We have that $\mcc^9(\tau_{1,2,3,4};P_0,\dotsc,P_8) \cap \Lambda_4$ has type $\lbm 4 &3\\ 1 &10 \rem$ inside $\Lambda_4$ and
\begin{equation*}
	\dim\lp \mcc^9(\tau_{1,2,3,4};P_0,\dotsc,P_8) \cap \Lambda_4 \rp \geq 20-11 \geq 9.
\end{equation*}

As a result, we can determine a rational point $P_9$ of $\mcc^9(\tau_{1,2,3,4};P_0,\dotsc,P_8) \setminus L(P_0,\dotsc,P_8)$ over an extension $L_6/L_5$ of degree at most 236196 and it follows that $(P_0,\dotsc,P_9)$ is a 9-polar point of $\tau_{1,2,3,4}$.

\end{proof}


\subsection{New Bounds from Moduli Spaces}\label{subsec:NewBoundsFromModuliSpaces}

In \cite{Wolfson2021}, Wolfson uses moduli space methods to determine $k$-planes on intersections of hypersurfaces over extensions of scalars of bounded resolvent degree. Wolfson then applies these methods to obtain upper bounds on $\RD(n)$. We will use a similar construction (Theorem \ref{thm:Determining A Point On Tau(d+k)}) and thus begin by giving an example of Wolfson's process in Example \ref{ex:WolfsonsMethod}. To do so, we must introduce additional language and notation.

\begin{definition}\label{def:ParameterAndModuliSpacesOfHypersurfaces} \textbf{(Parameter and Moduli Spaces of Hypersurfaces)}\\
	Fix $d \geq 2$. The \textbf{parameter space of degree $d$ hypersurfaces} in $\PP_\CC^r$ is 
\begin{align*}
	\mch(d;r) \cong \PP_\CC^{\binom{r+d}{d}-1}
\end{align*}

\noindent
However, there is a natural action of $\PGL(\CC,r+1)$ on $\PP_\CC^r$ which identifies hypersurfaces which are projectively equivalent. The \textbf{parameter space of semi-stable, degree $d$ hypersurfaces} in $\PP_\CC^r$ is the largest proper Zariski open $U \subseteq \mch(d;r)$ which is $\PGL(\CC,r+1)$-invariant; we denote it by $\mcs(d;r)$. Consequently, the \textbf{coarse moduli space of semi-stable, degree $d$ hypersurfaces} in $\PP_\CC^r$ is thus
	\begin{align*}
		\mcm(d;r) := \mcs(d;r)/ \PGL(\CC,r+1).
	\end{align*}
	
	The semi-stable locus $\mcs(d;r)$ is a dense Zariski open of $\mch(d;r)$ which is $\PGL(\CC,r+1)$-invariant. It contains another dense, $\text{PGL}(\CC,r+1)$-invariant Zariski open $\mcs^\circ(d;r) \subseteq \mcs(d;r)$ which parametrizes the smooth hypersurfaces. The \textbf{coarse moduli space of smooth, degree $d$ hypersurfaces} in $\PP_\CC^r$ is
	\begin{align*}
		\mcm^\circ(d;r) := \mcs^\circ(d;r) / \text{ PGL}(\CC,r+1).
	\end{align*}
\end{definition}

Each of the above parameter and moduli spaces classifies certain objects. We will additionally need to refer to the spaces classifying these objects with an associated choice of $k$-plane.

\begin{definition}\label{def:ParameterAndModuliSpacesOfHypersurfacesWithkPlanes} \textbf{(Parameter and Moduli Spaces of Hypersurfaces with $k$-Planes)}\\
	Continuing with the notation of Definition \ref{def:ParameterAndModuliSpacesOfHypersurfaces} and recalling that $\GR(k,r)$ is the variety of $k$-planes in $\PP_\CC^r$, we denote the \textbf{parameter space of degree $d$ hypersurfaces with choice of $k$-plane} in $\PP_\CC^r$ by $\mch(d;r,k)$; it is the incidence variety
	\begin{align*}
		\mch(d;r,k) = \lb (V,L) \st L \subseteq V \rb \subseteq \mch(d;r) \times \GR(k,r).
	\end{align*}
	
	\noindent	
	Similarly, we write $\mcs(d;r,k), \ \mcm(d;r,k), \text{ and } \mcm^\circ(d;r,k)$ for the analogous spaces which additionally classify a choice of $k$-plane. They similarly arise as incidence varieties or as quotients of incidence varieties by $\PGL(\CC,r+1)$.
\end{definition}

We now expand these definitions from classifying hypersurfaces to intersections of hypersurfaces of type $(2,\dotsc,d)$. 

\begin{definition}\label{def:Parameter And Moduli Spaces Of Intersections Of Hypersurfaces} \textbf{(Parameter and Moduli Spaces of Intersections of Hypersurfaces)}\\
	The \textbf{parameter space of intersections of hypersurfaces of type $(2,\dotsc,d)$} in $\PP_\CC^r$ is
	\begin{align*}
		\mch(2,\dotsc,d;r) = \mch(2;r) \times \cdots \times \mch(d;r).
	\end{align*}
	
	The natural action of $\PGL(\CC,r+1)$ action on $\PP_\CC^r$ induces a diagonal action on $\mch(2,\dotsc,d;r)$ and the \textbf{parameter space of semi-stable intersections of hypersurfaces of type $(2,\dotsc,d)$} in $\PP_\CC^r$ is the largest proper Zariski open $U \subseteq \mch(2,\dotsc,d;r)$ which is $\text{PGL}(\CC,r+1)$-invariant; it is denoted by $\mcs(2,\dotsc,d;r)$.
	
The \textbf{moduli space of semi-stable intersections of hypersurfaces of type $(2,\dotsc,d)$} in $\PP_\CC^r$ is
	\begin{align*}
		\mcm(2,\dotsc,d;r) = \mcs(2,\dotsc,d;r) / \text{ PGL}(\CC,r+1).
	\end{align*}
	
	In analogy with Definition \ref{def:ParameterAndModuliSpacesOfHypersurfacesWithkPlanes}, we write $\mch(2,\dotsc,d;r,k), \ \mcs(2,\dotsc,d;r,k),$ and $\mcm(2,\dotsc,d;r,k)$ for the respective spaces which additionally classify a choice of $k$-plane; they analogously arise as incidence varieties or as quotients of incidence varieties by $\PGL(\CC,r+1)$, as well.
\end{definition}

\begin{remark}\label{rem:Parameter and Moduli Spaces as Schemes} \textbf{(Parameter and Moduli Spaces as Schemes)}\\
For the interested reader, we note that these spaces can be constructed as schemes via classical invariant theory, beginning with
\begin{align*}
		\mch(d;r) &= \proj\lp S^*\lp \CC[x_0,\dotsc,x_r]_{(d)}^\vee \rp \rp,\\
		\mcm(d;r) &= \proj\lp S^*\lp \CC[x_0,\dotsc,x_r]_{(d)}^\vee \rp^{\GL(\CC,r+1)} \rp.
	\end{align*}		
	
	\noindent
	The remaining spaces arise analogously from the same constructions as in the variety case.
\end{remark}

We will use the dimension of these parameter and moduli spaces, so we recall the dimensions in the case of hypersurfaces and give the dimensions in the case of intersections of hypersurfaces of type $(2,\dotsc,d)$. In the proof of Proposition \ref{prop:Type Of a kth Polar Cone of a type (1,...,d)}, we observed the combinatorial identity
	\begin{equation*}
		\SL_{i=0}^d \binom{r+i}{i} = \binom{r+d+1}{d}.
	\end{equation*}
	
\noindent
In Remark \ref{rem:DimensionOfParameterAndModuliSpaces} and Definition \ref{def:NotationForkPlanesLemma}, we will use the slight variation
\begin{equation}\label{eqn:Combinatorial Identity}
	\SL_{i=2}^d \binom{r+i}{i} = \binom{r+d+1}{d} - (r+2).
\end{equation}

\begin{remark}\label{rem:DimensionOfParameterAndModuliSpaces} \textbf{(Dimension of Parameter and Moduli Spaces)}\\
	For fixed $d,r \geq 1$, we have
	\begin{equation*}
		\dim\lp \mcs(d;r) \rp = \dim\lp \mch(d;r) \rp = \binom{r+d}{d} - 1.
	\end{equation*}
	
	\noindent
	When $\binom{r+d}{d}-(r+1)^2<0$, $\mcm(d;r)$ is empty. When $\binom{r+d}{d}-(r+1)^2 \geq 0$, we have
	\begin{equation*}
	\dim\lp \mcm^\circ(d;r) \rp = \dim\lp \mcm(d;r) \rp = \binom{r+d}{d} - (r+1)^2.
	\end{equation*}
	
	\noindent
	Similarly, for $d,r$ for which the following spaces are non-empty, we have
	\begin{align*}
	&\dim\lp \mcs\lp 2,\dotsc,d;r \rp \rp = \dim\lp \mch\lp 2,\dotsc,d;r \rp \rp = \lp \SL_{i=2}^d \binom{r+i}{i} \rp - (d-1),\\
	&\dim\lp \mcm\lp 2,\dotsc,d;r \rp \rp = \lp \SL_{i=2}^d \binom{r+i}{i} \rp - (r+1)^2 - (d-2).
	\end{align*}
	
	\noindent
	From equation (\ref{eqn:Combinatorial Identity}), we re-write these quantities as
	\begin{align*}
		&\dim\lp \mcs\lp 2,\dotsc,d;r \rp \rp = \dim\lp \mch\lp 2,\dotsc,d;r \rp \rp = \binom{r+d+1}{d} - (r+d+1),\\
		&\dim\lp \mcm\lp 2,\dotsc,d;r \rp \rp = \binom{r+d+1}{d} - (r+1)^2 - (r+d).
	\end{align*}
\end{remark}

We now examine how Wolfson's algorithm determines $k$-planes on hypersurfaces. A theorem of Waldron (Theorem \ref{thm:Waldron}) guarantees the existence of a $k$-plane on a degree $d$ hypersurface once the ambient dimension $r$ is above a given threshold:

\begin{theorem}\label{thm:Waldron} \textbf{(Theorem 1.6 of \cite{Waldron2008})}\\
	Fix $d \geq 3$. When $r$ and $k$ are such that
	\begin{equation*}
		(k+1)(r-k) - \binom{k+d}{d} \geq 0,
	\end{equation*}
	
	\noindent
	then the natural maps
	\begin{align*}
		\mch(d;r,k) &\ra \mch(d;r),\\
		\mcm^\circ(d;r,k) &\ra \mcm^\circ(d;r),
	\end{align*}
	
	\noindent
	are surjective.\footnote{This is not the entirety of Waldron's result; see Theorem 1.6 of \cite{Waldron2008} for more details.}
\end{theorem}

Definition 2.4 of \cite{Wolfson2021} shows that the extended Tschirnhaus complete intersections are defined over $\ZZ$. Thus, for suitable $d,r,k$, Waldron's theorem allows us to determine a $k$-plane on $\tau_d$ over an extension $L/K_n$ with
\begin{equation*}
	\RD(L/K_n) \leq \RD\lp \mch(d;r,k) \ra \mch(d;r) \rp \leq \dim\lp \mch(d;r) \rp.
\end{equation*}

Moreover, Theorem 2.12  of \cite{Wolfson2021} shows that $\tau_{1,2,3}$ (along with certain other extended Tschirnhaus complete intersections) are generically smooth. It follows that for suitable $r,k$, Waldron's theorem allows us to determine a $k$-plane on $\tau_3$ over an extension $L'/K_n$ with
\begin{equation*}
	\RD(L/K_n) \leq \RD\lp \mcm^\circ(3;r,k) \ra \mcm^\circ(3;r) \rp \leq \dim\lp \mcm^\circ(3;r) \rp.
\end{equation*} 

\noindent
Note that for a generically finite, dominant, rational map of $\CC$-varieties $Y \dashrightarrow X$, the upper bound
\begin{equation}\label{eqn:FW Lemma 2.5}
	\RD(Y \dashrightarrow X) \leq \dim(X),
\end{equation}

\noindent
follows directly from Definition \ref{def:ResolventDegreeBranchedCovers}; it is also included in Lemma 2.5 of \cite{FarbWolfson2019}.

\begin{example}\label{ex:WolfsonsMethod} \textbf{(Wolfson's Method)}\\
For $n \geq 1559$, we can determine an $8$-plane on $\tau_{1,2,3,4}$ over an extension $L/K_n$ with $\RD(L/K_n) \leq 78485029$.
\end{example}

\begin{proof}
First, $\tau_1$ is a hyperplane and so $\tau_1 \cong \PP_{K_n}^{1557}$. Next, $\tau_2$ is a quadric hypersurface in $\tau_1$ and thus it is known classically that there is a 778-plane $\Lambda_2 \subseteq \tau_{1,2}$ over an iterated quadratic extension $L_1/K_n$ (see \cite{Wolfson2021} for details).  $\tau_{1,2,3} \cap \Lambda_2$ is a cubic hypersurface in $\Lambda_2$ and
	\begin{equation*}
		(63+1)(778-63) - \binom{63+3}{3} = 0,
	\end{equation*}
	
	\noindent
	hence we can determine a 63-plane $\Lambda_3 \subseteq \tau_{1,2,3} \cap \Lambda_2$ over an extension $L_2/L_1$ with
	\begin{equation*}
		\RD(L_2/L_1) \leq \RD\lp \mcm^\circ(3;778,63) \ra \mcm^\circ(3;778) \rp
		\leq \dim\lp \mcm^\circ(3;778) \rp
		= 78485029.
	\end{equation*}
	
	\noindent
	 Finally, $\tau_{1,2,3,4} \cap \Lambda_3$ is a quartic hypersurface in $\Lambda_3$ and
	\begin{equation*}
		(8+1)(63-8) - \binom{8+4}{4} = 0,
	\end{equation*}
	
	\noindent
	hence we can determine an 8-plane $\Lambda_4 \subseteq \tau_{1,2,3,4} \cap \Lambda_3$ over an extension $L_3/L_2$ with
	\begin{equation*}
		\RD(L_3/L_2) \leq \RD\lp \mch(4;63,8) \ra \mch(4;63) \rp
		\leq \dim\lp \mch(4;63) \rp
		= 766479.
	\end{equation*}
\end{proof}

We next provide an example which highlights the current limitations of using iterated polar cones for determining additional upper bounds on $\RD(n)$. 

\begin{example}\label{ex:Wolfson Computation For tau(1,2,3,4,5)} \textbf{(Wolfson Method for a 9-Plane on $\tau_{1,2,3,4,5}$)}\\
We can determine a 9-plane on $\tau_{1,2,3,4,5}$ over an extension $L/K_n$ with
\begin{equation*}
	\RD(L/K_n) \leq 3298353885918738132194252727911 \lp \ \approx 3 \cdot 10^{30} \ \rp,
\end{equation*}

\noindent
as long as
\begin{align*}
	r \geq 54097786526 \lp \ \approx 5 \cdot 10^{10} \rp.
\end{align*}				
\end{example}

	The core construction is the same as Example \ref{ex:WolfsonsMethod}. In Wolfson's notation, observe that $54097786526 = \psi(5,9)_4 + 1$ and
\begin{equation*}
	\dim\lp \mcm^\circ(3;\psi(5,9)_3) \rp = 3298353885918738132194252727911 \lp \ \approx 3 \cdot 10^{30} \ \rp.
\end{equation*}
		
\noindent
See the proof of Theorem 5.6 of \cite{Wolfson2021} for details.

\begin{remark} \textbf{(Iterated Polar Cone Comparison)}\\
	Observe that a $9^{th}$ polar cone $\mcc^9(\tau_{1,2,3,4,5};P_0,\dotsc,P_8)$ of $\tau_{1,2,3,4,5}$ has type $\lbm 5 &4 &3 &2 &1\\ 1 &10 &55 &220 &715 \rem$. Thus, even for $n$ large enough so that we may determine a suitable 66-plane $\Lambda$ on $\mcc^9(\tau_{1,2,3,4,5};P_0,\dotsc,P_8)_1 \cap \mcc^9(\tau_{1,2,3,4,5};P_0,\dotsc,P_8)_2$ (the intersection of the 220 quadrics and 715 hyperplanes defining $\mcc^9(\tau_{1,2,3,4,5};P_0,\dotsc,P_8)$), it still follows that
	\begin{align*}
		\deg\lp \mcc^9(\tau_{1,2,3,4,5},P_8) \cap \Lambda \rp &= 5 \cdot 4^{10} \cdot 3^{55},\\
		&= 914616279415496004448658427740160 &\lp \ \approx 9 \cdot 10^{32} \ \rp,\\
		&> 3298353885918738132194252727911 &\lp \ \approx 3 \cdot 10^{30} \ \rp,\\
		&= \dim\lp \mcm^\circ(3;\psi(5,9)_3) \rp.
	\end{align*}
	
	Indeed, the degree of the analogous intersection grows exponentially in $k$ whereas the required resolvent degree for Wolfson's method grows polynomially in $k$. Similar issues arise for $k$-planes on intersections of hypersurfaces of type (1,2,3,4) for $k \geq 33$. 
\end{remark}

In \cite{DebarreManivel1998}, Debarre and Manivel give an explicit combinatorial condition for an intersection of hypersurfaces in $\PP_\CC^r$ to contain a $k$-plane:

\begin{theorem}\label{thm:DebarreManivel} \textbf{(Theorem 2.1 of \cite{DebarreManivel1998})}\\
	Let $V \subseteq \PP_\CC^r$ be an intersection of hypersurfaces of type $\lbm d &d-1 &\cdots &2 &1\\ \ell_d &\ell_{d-1} &\cdots &\ell_2 &\ell_1 \rem$ which is not a quadric hypersurface. When $r$ and $k$ are such that
	\begin{align*}
		(k+1)(r-k) - \SL_{i=1}^d \ell_i \binom{k+i}{i} \geq 0,
	\end{align*}
	
	\noindent
	then the natural maps
	\begin{align*}
		\mch\lp 2,\dotsc,d; r,k\rp \ra \mch\lp 2,\dotsc,d;r \rp,\\
		\mcm\lp 2,\dotsc,d; r,k\rp \ra \mcm\lp 2,\dotsc,d;r \rp,
	\end{align*}
	
	\noindent
	are surjective. \footnote{This is not the entirety of Debarre and Manivel's result; see Theorem 2.1 of \cite{DebarreManivel1998} for more details.}
\end{theorem}

We will use this result to deal with $k$-planes on intersections of hypersurfaces directly (instead of using a recursive process on the set of defining hypersurfaces) to obtain new thresholds for upper bounds on $\RD(n)$ for $n \geq 348,489,068,134$ in Theorem \ref{thm:Determining A Point On Tau(d+k)}.

Similarly to Wolfson, we can use the upper bound
\begin{equation*}
	\RD\lp \mcm(2,\dotsc,d;r,k) \ra \mcm(2,\dotsc,d;r) \rp\leq \dim\lp \mcm(2,\dotsc,d;r) \rp,
\end{equation*}

\noindent
once we establish that the Tschirnhaus complete intersections $\tau_{1,\dotsc,m}$ are semi-stable in Lemma \ref{lem:kPlanes On An Intersection Of Type (2,...,d)} and Proposition \ref{prop:Semi-Stability of Tschirnhaus Complete Intersections}.

\begin{definition}\label{def:NotationForkPlanesLemma} \textbf{(Notation for Lemma \ref{lem:kPlanes On An Intersection Of Type (2,...,d)} )}\\
	To succinctly state Lemma \ref{lem:kPlanes On An Intersection Of Type (2,...,d)} and its applications, we define a new set function 
	\begin{align*}
		\vartheta:\ZZ_{\geq 3} \times \ZZ_{\geq 1} \ra \ZZ_{\geq 1},
	\end{align*}	
	
	\noindent	
	where $\vartheta(d,k)$ is the smallest positive integer $r$ such that
	\begin{align*}
		(k+1)(r-k) - \SL_{i=2}^d \binom{k+i}{i} \geq 0.
	\end{align*}	
	
	Using the combinatorial identity in equation (\ref{eqn:Combinatorial Identity}), we can equivalently write
	\begin{align*}
		\vartheta(d,k) = k + \left\lceil \frac{1}{k+1} \lp \binom{k+d+1}{d} - (k+2) \rp \right\rceil.
	\end{align*}	
\end{definition}

\begin{lemma}\label{lem:kPlanes On An Intersection Of Type (2,...,d)}  \textbf{($k$-Planes on an Intersection of Hypersurfaces of Type $(2,\dotsc,d)$)}\\
Let $d \geq 3$, and $V \in \mcs(2,\dotsc,d;r)(K)$ for some $\CC$-field $K$. For any $k \geq 1$, if $r \geq \vartheta(d,k)$, then we can determine a $k$-plane on $V$ over an extension $L/K$ with
\begin{align*}
	\RD(L/K) \leq \dim\lp \mcm(2,\dotsc,d;\vartheta(d,k) \rp.
\end{align*}
\end{lemma}

\begin{proof}
	First, observe that it suffices to prove the case where $r=\vartheta(d,k)$ by restriction. Theorem \ref{thm:DebarreManivel} then yields that $V$ contains a $k$-plane. We identify $V$ with $\AAA_K^0 \ra \mcm(2,\dotsc,d;\vartheta(d,k))$ and note that the resolvent degree of determining a $k$-plane is exactly the resolvent degree of the map
\begin{equation*}
	\pi_K:\AAA_K^0 \times_{\mcm(2,\dotsc,d;\vartheta(d,k))} \mcm(2,\dotsc,d;\vartheta(d,k),k) \ra \AAA_K^0
\end{equation*}

\noindent
determined by the pullback square
\[
\begin{tikzcd}
	\AAA_K^0 \times_{\mcm(2,\dotsc,d;\vartheta(d,k))} \mcm(2,\dotsc,d;\vartheta(d,k),k) \arrow{dd}[swap]{\pi_K} \arrow{rr}{\pi_\mcm} & &\mcm(2,\dotsc,d;\vartheta(d,k),k) \arrow{dd}\\
	\\
	\AAA_K^0 \arrow{rr} & &\mcm(2,\dotsc,d;\vartheta(d,k)).
\end{tikzcd}
\]

\noindent
However, Lemma 2.5 of \cite{FarbWolfson2019} yields that
\begin{align*}
	\RD\lp \pi_K \rp
	\leq \RD\lp \mcm(2,\dotsc,d;\vartheta(d,k),k) \ra \mcm(2,\dotsc,d;\vartheta(d,k)) \rp
	&\leq \dim\lp \mcm(2,\dotsc,d;\vartheta(d,k)) \rp.
\end{align*}
\end{proof}

\begin{prop}\label{prop:Semi-Stability of Tschirnhaus Complete Intersections} \textbf{(Semi-Stability of Tschirnhaus Complete Intersections)}\\
For each $d \geq 3$ and $n \geq d+2$, $\tau_{1,\dotsc,d} \in \mcs(2,\dotsc,d;n-2)(K_n)$.
\end{prop}

\begin{proof} First, note that $\tau_1 \subseteq \PP_{K_n}^{n-1}$ is a hyperplane, so we can consider $\tau_{1,\dotsc,d} \subseteq \tau_1 \cong \PP_{K_n}^{n-2}$, e.g. as a $K_n$-point of $\mcs(2,\dotsc,d;n-2)$.
	
	From the definition of $\mcs(2,\dotsc,d;n-2)$, it suffices to show that there is some $\PGL(K_n,n-1)$-invariant polynomial in $K_n[x_0,\dotsc,x_{n-1}]$ which does not vanish at $\tau_{1,\dotsc,d}$. Theorem 2.12 of \cite{Wolfson2021} yields that $\tau_{1,2,3}$ is generically smooth, hence semi-stable. In particular, there is a $\text{PGL}(K_n,n-1)$-invariant polynomial $f(x_0,\dotsc,x_{n-1}) \in K_n[x_0,\dotsc,x_{n-1}]$ which does not vanish at $\tau_{1,2,3} \in \mch(2,3;n-2)(K_n)$. When pulled back to $\mch(2,\dotsc,d;n-2)(K_n)$ via the standard projection map, $f(x_0,\dotsc,x_{n-1})$ does not vanish at $\tau_{1,\dotsc,d}$ as well. 
\end{proof}

\begin{theorem}\label{thm:Determining A Point On Tau(d+k)} \textbf{(Determining a Point on $\tau_{1,\dotsc,d+k}$)}\\
	Fix $k,d \geq 1$. For $n \geq \vartheta(d,k)+3$, we can determine a point of $\tau_{1,\dotsc,d+k}^\circ \subseteq \mct^n$ over an extension $L/K_n$ with
	\begin{align*}
		\RD(L/K_n) \leq \max\lb \dim\lp \mcm(2,\dotsc,d;\vartheta(d,k)) \rp, \frac{(d+k)!}{d!} \rb.
	\end{align*}
\end{theorem}

\begin{proof}
	By restriction, it suffices to prove the case where $n = \vartheta(d,k)+3$. As such, we work in $\PP_{K_n}^{\vartheta(d,k)+2}$. We then pass to a hypersurface $H$ which does not contain $[1:0:\cdots:0]$ and $\tau_1 \cap H \cong \PP_{K_n}^{\vartheta(d,k)}$ in this hypersurface. Hence, Lemma \ref{lem:kPlanes On An Intersection Of Type (2,...,d)} and Proposition \ref{prop:Semi-Stability of Tschirnhaus Complete Intersections} yield that we can determine a $k$-plane $\Lambda \subseteq \tau_{1,\dotsc,d}$ over an extension $L_1/K_n$ with
	\begin{align*}
		\RD(L_1/K_n) \leq \dim\lp \mcm(2,\dotsc,d;\vartheta(d,k)) \rp.
	\end{align*}

\noindent
Thus, $\deg\lp \tau_{1,\dotsc,d+k} \cap \Lambda \rp = \frac{(d+k)!}{d!}$ and we can determine a point $Q$ of $\tau_{1,\dotsc,d+k} \cap \Lambda \subseteq \tau_{1,\dotsc,d+k}^\circ$ over an extension of degree at most $\frac{(d+k)!}{d!}$.
\end{proof}

\begin{remark}\label{rem:WolfsonsFunctionF}  \textbf{(Wolfson's Function $F$)}\\
	In Definition 5.4 and Theorem 5.6 of \cite{Wolfson2021}, Wolfson introduces a monotone increasing function $F(r)$ such that $\RD(n) \leq n-r$ for $n \geq F(r)$; we will define this function explicitly in Section \ref{sec:ComparisonWithPriorBounds}. To remain consistent with the notation in this paper, we write $F$ as a function of $m$.
\end{remark}

	Wolfson shows shows that $\lim\limits_{m \ra \infty} \frac{B(m)}{F(m)} = \infty$, where $B(m)=(m-1)!+1$ was the previous best bounding function, established in \cite{Brauer1975}. We construct a similar function $G(m)$ and prove that $G$ similarly improves on $F$ (Theorem \ref{thm:ComparisonWithF}).

\begin{definition}\label{def:Definition of G} \textbf{(Defining the Function $G(m)$)}\\
We define $\varphi:\ZZ_{\geq 15} \times \ZZ_{\geq 1} \ra \ZZ_{\geq 1}$ by
\begin{align*}
	\varphi(d,k) = \max\lb \frac{(d+k)!}{d!}, \dim\lp \mcm(2,\dotsc,d;\vartheta(d,k)) \rp \rb.
\end{align*}

\noindent
Next, we define $G:\ZZ_{\geq 1} \ra \ZZ_{\geq 1}$  for $2 \leq m \leq 14$ by
\begin{center}
	\begin{tabular}{|c|ccccc|ccccc|}
		\hline
		$m$ &1 &2 &3 &4 &5 &6 &7 &8 &9 &10\\
		\hline
		$G(m)$ &2 &3 &4 &5 &9 &21 &109 &325 &1681 &15121\\
		\hline
	\end{tabular}
	
	\vspace{12pt}	
	
	\begin{tabular}{|c|cccc|}
		\hline
		$m$ &11 &12 &13 &14\\
		\hline
		$G(m)$ &151,201 &1,663,201 &19,958,401 &259,459,201\\
		\hline
	\end{tabular}
\end{center}

\noindent
and for $m \geq 15$ by
\begin{align*}
	G(m) = 1+\min \lb \varphi(d,m-d-1) \st 4 \leq d \leq m-1 \rb.
\end{align*}
\end{definition}

\begin{theorem}\label{thm:UpperBoundsonRD(n)} \textbf{(Upper Bounds on $\RD(n)$)}\\
	For each $m \geq 1$ and all $n \geq G(m)$, $\RD(n) \leq n-m$.
\end{theorem}

\begin{proof} \textbf{(Proof of Theorem \ref{thm:UpperBoundsonRD(n)})}\\
	The claim for $1 \leq m \leq 5$ is classical and is covered in \cite{Wolfson2021}. The case $m=6$ is Theorem \ref{thm:The n-6 Bound} and the cases of $7 \leq m \leq 14$ are Theorem \ref{thm:The n-7,...,n-14 Bounds}.
	
	Now, fix $m \geq 15$. For each $4 \leq d \leq m-1$, Theorem \ref{thm:Determining A Point On Tau(d+k)} yields that we can determine a point of $\tau_{1,\dotsc,m-1}^\circ \subseteq \mct^n$ over an extension of scalars $L/K_n$ with $\RD(L/K_n) \leq \varphi(d,m-d-1)$ when $n \geq \vartheta(d,m-d-1)+3$. Note that
	\begin{align*}
		\vartheta(d,m-d-1)+3 &< \binom{\vartheta(d,m-d-1)+d+1}{d} - (\vartheta(d,m-d-1)+1)^2 - (\vartheta(d,m-d-1)+d)\\
		&= \dim\lp \mcm(2,\dotsc,d;\vartheta(d,m-d-1)) \rp\\
		&\leq \varphi(m,m-d-1). 
	\end{align*}		

\noindent
It suffices to minimize over all such $d$ and so the definition of $G(m)$ and Remark \ref{rem:RDBoundsFromTschirnhausTransformations} yield the theorem.
\end{proof}

In Section \ref{sec:ComparisonWithPriorBounds}, we pin down how $G(m)$ improves on $F(m)$. First, we use $G(m)$ to establish an upper bound on $\RD(n)$ using elementary functions \`{a} la Brauer.


\subsection{Upper Bounds on the Bounding Function $G(m)$}\label{subsec:Upper Bounds on the Bounding function G(m)}

While $G(m)$ is simpler than $F(m)$, its definition does not immediately yield a description in terms of elementary functions. We state such a description now and use the remainder of the subsection to obtain this approximation.

\begin{theorem}\label{thm:UpperBoundOnGrowthRateofRD(n)} \textbf{(Upper Bound on the Growth Rate of $\RD(n)$)}\\
	For every positive integer $d \geq 4$, $G(2d^2+7d+6) \leq \frac{(2d^2+7d+5)!}{d!}$. Hence, for $n \geq \frac{\lp 2d^2 + 7d + 5 \rp!}{d!}$, it follows that
	\begin{align*}
		\RD(n) \leq n-2d^2-7d-6.
	\end{align*}
\end{theorem}

To prove Theorem \ref{thm:UpperBoundOnGrowthRateofRD(n)}, we establish a simple criterion for $m$, in terms of $d$, so that we can conclude that $G(m) < \frac{(m-1)!}{d!}$ when that criterion is met.

Recall that $\vartheta(d,k)$ is defined such that an intersection of hypersurfaces of type $(1,\dotsc,d)$ in $\PP^r$ contains a $k$-plane when $r \geq \vartheta(d,k)$. We begin by approximating $\vartheta(d,m-d-1)$ above. 

\begin{lemma}\label{lem:UpperBoundOnVartheta} \textbf{(Upper Bound on $\vartheta$)}\\
	Fix $m > d  \geq 4$. Then
	\begin{align*}
		\vartheta(d,m-d-1) \leq m-d-2 + \binom{m}{d}.
	\end{align*}
\end{lemma}

\begin{proof}
We first recall Definition \ref{def:NotationForkPlanesLemma}:
	\begin{align*}
		\vartheta(d,k) = k + \left\lceil \frac{1}{k+1} \lp \binom{k+d+1}{d} - (k+2) \rp \right\rceil.
	\end{align*}
	
	\noindent
	By using the identification $k = m-d-1$, we observe that
	\begin{align*}
		\vartheta(d,m-d-1) &= (m-d-1) + \left\lceil \frac{1}{(m-d-1)+1} \lp \binom{(m-d-1)+d+1}{d} - ((m-d-1)+2) \rp \right\rceil,\\
		&= (m-d-1) + \left\lceil \frac{1}{m-d} \lp \binom{m}{d} - (m-d+1) \rp \right\rceil,\\
		&\leq m-d-2 + \binom{m}{d}.
	\end{align*}
\end{proof}

\begin{corollary}\label{cor:UpperBoundOnParameterSpaceDimension} \textbf{(Upper Bound on Parameter Space Dimension)}\\
	Fix $m > d \geq 4$. Then,
	\begin{align*}
		\dim\lp \mch(2,\dotsc,d;\vartheta(d,m-d-1)) \rp \leq \binom{m-1+\binom{m}{d}}{d} - \lp m-1 + \binom{m}{d} \rp.
	\end{align*}
\end{corollary}

\begin{proof}
	Remark \ref{rem:DimensionOfParameterAndModuliSpaces} established that
	\begin{align*}
		\dim\lp \mch(2,\dotsc,d;\vartheta(d,m-d-1)) \rp \leq \binom{\vartheta(d,m-d-1)+d+1}{d} - (\vartheta(d,m-d-1)+d+1),
	\end{align*}
	
	\noindent
	which is non-decreasing in $\vartheta(d,m-d-1)$. Thus, Lemma \ref{lem:UpperBoundOnVartheta} yields
	\begin{align*}
	\dim\lp \mch(2,\dotsc,d;\vartheta(d,m-d-1)) \rp \leq \binom{m-1+\binom{m}{d}}{d} - \lp m-1 + \binom{m}{d} \rp.
	\end{align*}
\end{proof}

	We will now introduce a constant $C_d$ and give a bound on $\log\lp \frac{C_d}{d+1} \rp$, both of which will be useful in the proof of Lemma \ref{lem:TheVarphiCondition}. We also remind the reader that every use of $\log$ in this paper refers to the base $e$ logarithm.

\begin{definition}\label{def:ConstantForGComputationLemma} \textbf{(Constant for the Proof of Lemma \ref{lem:TheVarphiCondition} )}\\
For each $d \geq 4$, we set
\begin{align*}
	C_d := \max\lb \binom{d+1}{i} \st 0 \leq i \leq d+1 \rb.
\end{align*}
\end{definition}

\begin{lemma}\label{lem:BoundOnLogCd} \textbf{(Bound on $\log(C_d)$)}\\
	For each $d \geq 1$, it follows that
	\begin{align*}
		\log\lp \frac{C_d}{d+1} \rp \leq d+\frac{3}{2}.
	\end{align*}
\end{lemma}

We will frequently use Stirling's approximations for factorials, including in the proof of Lemma \ref{lem:BoundOnLogCd}, and thus state the version we use explicitly (a stronger version of which can be found in \cite{Robbins1955}):\\

\noindent
Let $a$ be an positive integer. Then,
\begin{equation}\label{eqn:Stirling's Approximation}
	\sqrt{2\pi} a^{a+\frac{1}{2}} e^{-a} \leq a! \leq a^{a+\frac{1}{2}} e^{1-a}.
\end{equation}

\begin{proof} \textbf{(Proof of Lemma \ref{lem:BoundOnLogCd})}\\
	Our proof depends on the parity of $d+1$; we begin with the case where $d+1 = 2\ell$ is even. Then, $C_d = \binom{2\ell}{\ell} = \frac{(2\ell)!}{(\ell!)^2}$ and Stirling's approximation yields	
	\begin{align*}
		C_d \leq \frac{ (2\ell)^{2\ell+\frac{1}{2}}e^{1-2\ell} }{\lp \sqrt{2\pi}\ell^{\ell+\frac{1}{2}}e^{-\ell}  \rp^2} = \frac{ 2^{d+\frac{1}{2}}e }{\pi \sqrt{\ell}}.
	\end{align*}
	
	\noindent
	Consequently,
	\begin{align*}
		\log\lp \frac{C_d}{d+1} \rp \leq \log\lp \frac{ 2^{d-\frac{1}{2}}e }{\pi \sqrt{\ell}(d+1)} \rp \leq \log\lp \frac{e}{\pi(d+1)\sqrt{\ell}} \rp + \log\lp e^{d+\frac{1}{2}} \rp \leq d + \frac{1}{2}.
	\end{align*}
	
	\noindent
	When $d+1 = 2\ell+1$ is odd, we observe
	\begin{align*}
		C_d = \binom{2\ell+1}{\ell} \leq \binom{2\ell+2}{\ell+1} \leq \frac{ 2^{d+\frac{3}{2}}e }{\pi \sqrt{\ell+1}},
	\end{align*}
	
	\noindent
	and thus
	\begin{align*}
		\log\lp \frac{C_d}{d+1} \rp \leq \log\lp \frac{ 2^{d+\frac{3}{2}}e }{\pi \sqrt{\ell+1}} \rp \leq d+\frac{3}{2}.
	\end{align*}
\end{proof}

Recall that for each $1 \leq d \leq m-1$ and when $n$ is large enough, we can determine an $(m-d-1)$-plane $\Lambda \subseteq \tau_{1,\dotsc,d}^\circ$ over an extension $L_1/K_n$ with
\begin{align*}
	\RD(L_1/K_n) \leq \dim\lp \mcm(2,\dotsc,d;\vartheta(d,m-d-1)) \rp.
\end{align*}

\noindent
In such a case, we can determine a point of $\tau_{1,\dotsc,m-1} \cap \Lambda$ over an extension $L_2/L_1$ of degree $\frac{(m-1)!}{d!}$. Hence, we set
\begin{align*}
	\varphi(d,m-d-1) = \max\lb \frac{(m-1)!}{d!}, \dim\lp \mcm(2,\dotsc,d;\vartheta(d,m-d-1)) \rp + m \rb
\end{align*}

\noindent
in Definition \ref{def:Definition of G}. We next give a condition relating $\varphi(d,m-d-1)$ and $\varphi(d+1,m-d-2)$.

\begin{lemma}\label{lem:TheVarphiCondition} \textbf{(The $\varphi$ Condition)}\\
Fix $d \geq 4$. For all $m \geq 2d^2 + 7d + 6$, it follows that
\begin{equation}\label{eqn:Varphi Condition}
	\varphi(d+1,m-d-2) < \varphi(d,m-d-1). 
\end{equation}
\end{lemma}

\begin{proof} For any such $d$ and $m$, it is always true that
	\begin{align*}
		\frac{(m-1)!}{(d+1)!} < \frac{(m-1)!}{d!}.
	\end{align*}
	
	As a result, to conclude (\ref{eqn:Varphi Condition}), we need only show that
	\begin{equation}\label{eqn:varphi1}
		\frac{(m-1)!}{d!} > \binom{m-1+\binom{m}{d+1}}{d+1},
	\end{equation}
	
	\noindent
	since
	\begin{equation*}
		\binom{m-1+\binom{m}{d+1}}{d+1} > \dim\lp \mch(2,\dotsc,d;\vartheta(d,m-d-1)) \rp > \dim\lp \mcm(2,\dotsc,d;\vartheta(d,m-d-1)) \rp.
	\end{equation*}
	
	\noindent	
	Observe that
	\begin{equation}\label{eqn:varphi2}
		\binom{m-1+\binom{m}{d+1}}{d+1} = \frac{ \lp m-1+\binom{m}{d+1} \rp! }{(d+1)! \lp m-d-2+\binom{m}{d+1} \rp}
		= \frac{1}{(d+1)!} \PL_{i=1}^{d+1} \lp m-i+\binom{m}{d+1} \rp.
	\end{equation}
	
	\noindent
	Next, we approximate $\binom{m}{d+1}$:	
	\begin{equation}\label{eqn:varphi3}
		\binom{m}{d+1}  = \frac{1}{(d+1)!} \cdot \frac{m!}{(m-d-1)!} = \frac{1}{(d+1)!} \PL_{j=0}^{d+1}(m-j) \leq \frac{1}{(d+1)!}m^{d+2} \leq m^{d+2}.
	\end{equation}
	
	By substituting the inequality (\ref{eqn:varphi3}) into equation (\ref{eqn:varphi2}) and using that $m-i \leq m$, we obtain the approximation
	\begin{align*}
		\binom{m-1+\binom{m}{d+1}}{d+1} \leq \frac{1}{(d+1)!} \PL_{i=1}^{d+1} \lp m +m^{d+2} \rp,
	\end{align*}
	
	\noindent
	which, after substituting into inequality (\ref{eqn:varphi1}) and multiplying both sides by $d!$, yields the sufficient condition
	\begin{equation}\label{eqn:varphi4}
		(m-1)! > \frac{1}{d+1}\PL_{i=1}^{d+1} \lp m+m^{d+2} \rp.
	\end{equation}
	
	Notice that main term of the right side of inequality (\ref{eqn:varphi4}) is of the form $\PL_{i=1}^{d+1} (a+b) = (a+b)^{d+1}$, hence applying the binomial theorem yields
	\begin{align*}
		\PL_{i=1}^{d+1}\lp m+m^{d+2} \rp &= \SL_{i=0}^{d+1} \binom{d+1}{i} \lp m^{d+2} \rp^i m^{d+1-i},\\
		&= \SL_{i=0}^{d+1} \binom{d+1}{i} m^{di+2i} (m-1)^{d+1-i},\\
		&= \SL_{i=0}^{d+1} \binom{d+1}{i} m^{(d+1)i+d+1},\\
		&= m^{d+1}\SL_{i=0}^{d+1} \binom{d+1}{i} \lp m^{d+1} \rp^i.
	\end{align*}
	
	\noindent
	However, for any positive integer $a$ and $x \geq a$ it follows from induction $\SL_{i=0}^{a+1} x^i \leq 2x^{a+1}.$	Applying this to $(m-1)\SL_{i=0}^{d+1} \binom{d+1}{i} \lp (m-1)^d \rp^i$ and recalling $C_d = \max\lb \binom{d+1}{i} \st 0 \leq i \leq d+1 \rb$ yields
	\begin{equation}\label{eqn:varphi5}
		m^{d+1}\SL_{i=0}^{d+1} \binom{d+1}{i} \lp m^{d+1} \rp^i \leq m^{d+1}C_d\lp 2 \lp m^{d+1} \rp^{d+1} \rp.
	\end{equation}		
	
	\noindent
	Substituting inequality (\ref{eqn:varphi5}) into inequality (\ref{eqn:varphi4}) and simplifying, we obtain the condition
	\begin{align*}
		(m-1)! > \frac{2}{d+1}C_dm^{d^2+3d+2}.
	\end{align*}
	
	\noindent
	Next, we apply Stirling's approximation and re-arrange terms to arrive at the condition
	\begin{equation}\label{eqn:varphi6}
		\frac{ (m-1)^{m-\frac{1}{2}} }{e^{m-1}} > \frac{2C_d}{\sqrt{2\pi}(d+1)}m^{d^2+3d+2}.
	\end{equation}
	
	\noindent
	Observe that $a^{a-1} < (a-1)^{a-\frac{1}{2}}$ for positive integers $a \geq 8$. By requiring $m > 8$, we need only consider when
	\begin{equation}\label{eqn:varphi6'}
		\frac{ m^{m-1} }{e^{m-1}} > \frac{2C_d}{\sqrt{2\pi}(d+1)} m^{d^2+3d+2}.
	\end{equation}
	
	\noindent
	Multiplying both sides of inequality (\ref{eqn:varphi6'}) by $e^{d^2+3d+2}$, dividing both sides by $m^{d^2+3d+2}$, and simplifying, we arrive at the condition
	\begin{equation}\label{eqn:varphi7}
		\lp \frac{m}{e} \rp^{m-d^2-3d-3} > \frac{2C_de^{d^2+3d+2}}{\sqrt{2\pi}(d+1)}.
	\end{equation}
	
	\noindent
	We take $\log$ of both sides of inequality (\ref{eqn:varphi7}), which yields
	\begin{equation}\label{eqn:varphi8}
		\lp m - d^2 - 3d - 3 \rp \log\lp \frac{m}{e} \rp > \log\lp \frac{2C_de^{d^2+3d+3}}{\sqrt{2\pi}(d+1)} \rp.
	\end{equation}
	
	\noindent
	Requiring $m > e^2$, it suffices to consider
	\begin{equation}\label{eqn:varphi9}
		m - d^2 - 3d - 3 > \log\lp \frac{2}{\sqrt{2\pi}} \rp + \log\lp \frac{C_d}{d+1} \rp + \lp d^2 + 3d + 2 \rp.
	\end{equation}
	
	\noindent
	We note $\log\lp \frac{2}{\sqrt{2\pi}} \rp < 0$ and apply the bound on $\log\lp \frac{C_d}{d+1} \rp \leq d+ \frac{3}{2}$ from Lemma \ref{lem:BoundOnLogCd} to obtain the sufficient condition
	\begin{equation*}
		m - d^2 - 3d - 3 > \lp d+\frac{3}{2} \rp + \lp d^2 + 3d + 2 \rp,
	\end{equation*}
	
	\noindent
	which re-arranges to yield our initial supposition
	\begin{equation}\label{eqn:varphi10}
		m \geq 2d^2 + 7d + 6.
	\end{equation}
	
	\noindent
	Finally, note that for all $d \geq 4$, $2d^2 + 7d + 6 > e^2 > 8,$ hence the requirements $m \geq 8$, $m > e^2$ used in the proof are rendered superfluous.
\end{proof} 

\begin{corollary}\label{cor:UpperBoundOnGrowthRateofG} \textbf{(Upper Bound on the Growth Rate of $G$)}\\
	For any $d \geq 4$ and $m \geq 2d^2 + 7d + 6$, it follows that
	\begin{align*}
		G(m) < \frac{(m-1)!}{d!}.
	\end{align*}
\end{corollary}

\begin{proof}
Recall Definition \ref{def:Definition of G}; in particular, for $m \geq 15$,
\begin{align*}
	G(m) = \min\lb \varphi(d,m-d-1) \st 15 \leq d \leq m-d-1 \rb + 1.
\end{align*}

Consequently, for any $d \geq 4$ and $m \geq 2d^2 + 3d$, we have
	\begin{align*}
		G(m) \leq \varphi(d+1,m-d-2)+1 < \frac{(m-1)!}{d!}
	\end{align*}
	
	\noindent
	from Lemma \ref{lem:TheVarphiCondition}.
\end{proof}

We now prove Theorem \ref{thm:UpperBoundOnGrowthRateofRD(n)}.

\begin{proof} \textbf{(Proof of Theorem \ref{thm:UpperBoundOnGrowthRateofRD(n)})}\\
Fix $d \geq 4$. From Corollary \ref{cor:UpperBoundOnGrowthRateofG}, we have that
\begin{equation*}
	G(m) \leq \frac{(2d^2+7d+5)!}{d!}
\end{equation*}

\noindent
for $m \geq 2d^2+7d+6$. From Theorem \ref{thm:UpperBoundsonRD(n)}, it follows that
\begin{align*}
	\RD(n) \leq n-2d^2-7d-6
\end{align*}

\noindent
for $n \geq \frac{(2d^2+7d+5)!}{d!}$. 
\end{proof}


\section{Comparison with Prior Bounds}\label{sec:ComparisonWithPriorBounds}

We now give a precise sense of how the bounds from $G(m)$ improve upon the bounds of $F(m)$. 

\begin{theorem}\label{thm:ComparisonWithF}	\textbf{(Comparison With $F(m)$)}\\
	Let $F$ be the function defined in Definition \ref{def:Wolfsons Functions} (which is originally Definition 5.4 of \cite{Wolfson2021}).
	\begin{enumerate}
		\item For every $m \geq 1$, $G(m) \leq F(m)$ with equality if and only if $m=1, 2, 3, 4, 5, 15,$ or $16$.
	\item $G(m)$ provides asymptotic improvements on $F(m)$, in the sense that
	\begin{align*}
		\lim\limits_{m \ra \infty} \frac{F(m)}{G(m)} = \infty.
	\end{align*}
	\end{enumerate}	  	
\end{theorem}

\begin{remark}
	Despite Theorem \ref{thm:ComparisonWithF}, $G(m)$ does not yield a strictly better bound on $\RD(n)$ for all $n$. As an example,
	\begin{align*}
		G(17) &= 348,489,068,134, &F(17) &= 871,782,912,001,\\
		G(18) &= 2,964,061,900,801 &F(18) &= 14,820,309,504,001,
	\end{align*}		
	so $\frac{F(17)}{G(17)} \sim 2.502$ and $\frac{F(18)}{G(18)} \sim 5.000$. However, for any integer $n$ between $F(17)$ and $G(18)$, $F$ and $G$ yield the same upper bound; namely, $\RD(n) \leq n-17$.
\end{remark}

The remainder of the section is spent proving Theorem \ref{thm:ComparisonWithF}. We will now define Wolfson's function $F(m)$. Our presentation will vary slightly from Wolfson's  and we refer the reader to Section 5 of \cite{Wolfson2021} for details of the construction.

\begin{definition}\label{def:Wolfsons Functions}  \textbf{(Wolfson's Functions)}\\
	Given $(d,k) \in \ZZ_{\geq 3} \times \ZZ_{\geq 1}$, set $\psi(d,k)_0 = k$. For $0 \leq i < d-2$, set
	\begin{align*}
		\psi(d,k)_{i+1} := \psi(d,k)_i + \left\lceil \frac{1}{\psi(d,k)_i+1} \cdot \binom{\psi(d,k)_i+d-i}{d-i} \right\rceil,
	\end{align*}	
	
	\noindent
	along with $\psi(d,k)_{d-1} := 2\psi(d,k)_{d-2}+1$.
	
	Additionally, define
	\begin{align*}
		\Phi(d,k) := \max\lb \frac{(d+k)!}{d!}, \dim\lp \mcm\lp 3;\psi(d,k)_{d-2} \rp \rp +d+k+1 \rb.
	\end{align*}
	
	\noindent
	Finally, for $m \leq 3$, set $F(m):=m+1$ and for $m \geq 4$, set
	\begin{align*}
		F(m) := 2 \left\lfloor \frac{1}{2} \lp \min\limits_{1 \leq d \leq m-2} \Phi(d,m-d-1) \rp \right\rfloor + 1.
	\end{align*}
\end{definition}

\begin{remark}\label{rem:Outline for Section 4}  \textbf{(Outline for Section \ref{sec:ComparisonWithPriorBounds})}\\
	Our goal is to establish a criterion for $m$, in terms of $d$, so that we can conclude $F(m) > \frac{(m-1)!}{d!}$ when this criterion is met. We do this by examining when
	\begin{align*}
		\Phi(d,m-d-1) < \Phi(d+1,m-d-2).
	\end{align*}
	
	\noindent
	We first show that for any $m \geq d+3$,
	\begin{equation}\label{eqn:ss42-1}
		\dim\lp \mcm\lp 3;\psi(d,m-d-1)_{d-2} \rp \rp \leq \dim\lp \mcm\lp 3;\psi(d+1,m-d-2)_{d-1} \rp \rp,
	\end{equation}
	
	\noindent
	which will then leave us to consider when
	\begin{equation}\label{eqn:ss43-2}
		\frac{(m-1)!}{d!} < \dim\lp \mcm\lp 3;\psi(d+1,m-d-2)_{d-1} \rp \rp.
	\end{equation}
\end{remark}

We begin by introducing auxiliary functions which will be useful for proving inequality (\ref{eqn:ss42-1}) holds for $m \geq d+3$.

\begin{definition}\label{def:AuxiliaryFunctionsPsi(d;i)}  \textbf{(Auxiliary Functions $\Psi(d;i)$)}\\
	For any $d \geq 2$, any $d-2 \geq i \geq 1$ and $x \in \ZZ_{\geq 1}$, we set
	\begin{align*}
		\Psi(d;i)(x) = x + \left\lfloor \frac{1}{x+1} \cdot \binom{x+d-i}{d-i} \right\rfloor.
	\end{align*}
\end{definition}

\begin{remark}\label{rem:Key Properties of Psi(d;i)} \textbf{(Key Properties of $\Psi(d;i)$)}\\
	The functions $\Psi(d,i)$ satisfy
	\begin{enumerate}
	\item $\Psi(d;i+1)\lp \psi(d,m-d-1)_{i} \rp = \psi(d,m-d-1)_{i+1}$ for $i \leq d-3$, and
	\item $\Psi(d+1;i+1) = \Psi(d,i)$.
	\end{enumerate}
\end{remark}

\begin{lemma}\label{lem:Psi(d;i)NonDecreasing}  \textbf{($\Psi(d;i)$ are Non-Decreasing)}\\
	For each $m \geq d+1$ with $d \geq 2$ and each $d-2 \geq i \geq 1$, the function $\Psi(d;i)(x)$ is non-decreasing. 
\end{lemma}

\begin{proof}
	First, observe that
	\begin{align*}
		\Psi(d;i)(x+1) - \Psi(d;i)(x) = 1 + \left\lfloor \frac{1}{x+2} \cdot \binom{x+1+d-i}{d-i} \right\rfloor - \left\lfloor \frac{1}{x+1} \cdot \binom{x+d-i}{d-i} \right\rfloor,
	\end{align*}
	
	\noindent
	thus
	\begin{align*}
	\Psi(d;i)(x+1) - \Psi(d;i)(x) \geq \frac{1}{x+2} \cdot \binom{x+1+d-i}{d-i} - \frac{1}{x+1} \cdot \binom{x+d-i}{d-i}.
	\end{align*}	
	
	Now, observe that
	\begin{align*}
		\frac{1}{x+2} \cdot \binom{x+1+d-i}{d-i} &= \frac{(x+1+d-i)!}{(x+2)(d-i)!(x+1)!} = \frac{(x+d-i)!}{(d-i)!(x+1)!} \lp \frac{x+d-i+1}{x+2} \rp,
	\end{align*}
	
	\noindent
	and
	\begin{align*}
		\frac{1}{x+1} \cdot \binom{x+d-i}{d-i} &= \frac{(x+d-i)!}{(x+1)(d-i)!x!} = \frac{(x+d-i)!}{(d-i)!(x+1)!}.
	\end{align*}
	
	\noindent
	Hence,
	\begin{align*}
	\Psi(d;i)(x+1) - \Psi(d;i)(x) \geq \frac{(x+d-i)!}{(d-i)!(x+1)!} \lp \frac{x+d-i+1}{x+2} - 1 \rp,
	\end{align*}
	
	\noindent
	and the right side is positive when $i \leq d-2$.
\end{proof}

\begin{lemma}\label{lem:Psi(d,k)d-2NonDecreasing} \textbf{($\psi(d,m-d-1)_{d-2}$ is Non-Decreasing in $d$)}\\
	Fix $m \geq 4$. Then,
	\begin{align*}
		\psi(2,m-3)_0 \leq \psi(3,m-4)_1 \leq \cdots \leq \psi(m-3,2)_{m-5} \leq \psi(m-2,1)_{m-4}.
	\end{align*}
	
	\noindent
	Furthermore, for each $2 \leq d \leq m-3$,
	\begin{align*}
		\dim\lp \mcm\lp 3;\psi(d,m-d-1)_{d-2} \rp \rp \leq \dim\lp \mcm\lp 3;\psi(d+1,m-d-2)_{d-1} \rp \rp.
	\end{align*}
\end{lemma}

\begin{proof}
	In light of Remark \ref{rem:Key Properties of Psi(d;i)} and Lemma \ref{lem:Psi(d;i)NonDecreasing}, it suffices to show
	\begin{align*}
		\psi(d+1,m-d-2)_1 \geq \psi(d,m-d-1)_0
	\end{align*}
	
	\noindent
	to prove the claim. However,
	\begin{align*}
		\psi(d+1,m-d-2)_1 &= m-d-2 + \left\lceil \frac{1}{m-d-1} \cdot \binom{m-2}{d-1} \right\rceil 
		\geq m-d-1 
		= \psi(d,m-d-1)_0.
	\end{align*}
\end{proof}

\begin{remark}
	Having proved Lemma \ref{lem:Psi(d,k)d-2NonDecreasing}, we now begin to work towards the second task outlined in Remark \ref{rem:Outline for Section 4}: establishing a criterion for $m$ in terms of $d$, which, when met, implies
	\begin{align*}
		\frac{(m-1)!}{d!} \leq \dim\lp \mcm\lp 3;\psi(d+1,m-d-2)_{d-1} \rp \rp.
	\end{align*}
	
	Next, we establish simple functions which we use to approximate $\psi(d,m-d-1)_{d-2}$ from below.
\end{remark}

\begin{definition}\label{def:AuxiliaryFunctionsVarrhoAndRho(d,i)} \textbf{(Auxiliary Functions $\Omega$ and $\omega(d,i)$)}\\
	For each $d \geq 4$ and $d-3 \geq i \geq 1$, define $\omega(d,i):\RR_{>0} \ra \RR_{>0}$ by
	\begin{align*}
		\omega(d,i)(x) = \frac{1}{(d-i)!}x^{d-i-1}.
	\end{align*}
	
	\noindent
	Similarly, for each pair $m \geq d$ with $d \geq 4$, define the function $\Omega$ by
	\begin{align*}
		\Omega(d,m) = \lp \omega(d,d-3) \circ \cdots \circ \omega(d,1) \rp (m-d-1).
	\end{align*}
\end{definition}

\begin{remark}\label{rem:Bounding Properties of Omega and omega} \textbf{(Bounding Properties of $\Omega$ and $\omega(d,i)$)}\\
	Observe that for each $i$,
	\begin{align*}
		\omega(d,i+1)(\psi(d,m-d-1)_i) \leq \psi(d,m-d-1)_{i+1}.
	\end{align*}
	
	\noindent
	In particular,
	\begin{align*}
		\Omega(d,m) \leq \psi(d,m-d-1)_{d-2}.
	\end{align*}
\end{remark}

\begin{example}
	Consider the cases when $d=5$ and $m=10, 100$. Then,
	\begin{align*}
		\Omega(5,10) &\sim 1.185, &\Omega(5,100) &\sim 1.996 \times 10^8,\\
		\psi(5,10)_3 &= 133, &\psi(5,100)_3 &\sim 3.633 \times 10^8.
	\end{align*}
\end{example}

\begin{lemma}\label{lem:ExplicitFormOfOmega} \textbf{(Explicit Form of $\Omega$)}\\
	For any $d \geq 4$, set
	\begin{equation*}
		\mfc_d := \PL_{i=3}^{d-1} \frac{1}{(i!)^{(i-2)!}}. 
	\end{equation*}
	
	\noindent
	For all $m \geq d+2$,
	\begin{equation*}
		\Omega(d,m) = \mfc_d (m-d-1)^{(d-2)!},
	\end{equation*}	
\end{lemma}

\begin{proof}
	We proceed by induction on $d$. When $d=4$, we have
	\begin{align*}
		\Omega(4,m) &= \rho(4,1)(m-5) = \frac{1}{3!}(m-5)^2 = \mfc_2(m-5)^{2!}.
	\end{align*}
	
	\noindent
	For arbitrary $d$, recall
	\begin{align*}
		\Omega(d,m) &= \lp \omega(d,d-3) \circ \cdots \circ \omega(d,2) \circ \omega(d,1) \rp (m-d-1),
	\end{align*}
	
	\noindent
	and
	\begin{align*}
		\Omega(d-1,m-1) &= \lp \omega(d-1,d-4) \circ \cdots \circ \omega(d-1,2) \circ \omega(d-1,1) \rp (m-d-1).
	\end{align*}
	
	\noindent
	By definition, however, $\omega(d+1,i+1) = \omega(d,i)$. Therefore,
	\begin{align*}
	\Omega(d,m) &= \lp \omega(d,d-3) \circ \cdots \circ \omega(d,2) \circ \omega(d,1) \rp (m-d-1),\\
	&= \lp \omega(d,d-3) \circ \cdots \circ \omega(d,2) \rp \lp \omega(d,1)(m-d-1) \rp,\\
	&= \lp \omega(d-1,d-4) \circ \cdots \circ \omega(d-1,1) \rp \lp \omega(d,1)(m-d-1) \rp,\\
	&= \Omega(d-1, \omega(d,1)(m-d-1)+d+1).
	\end{align*}
	
	\noindent
	By induction, we know that
	\begin{align*}
		\Omega(d-1, \omega(d,1)(m-d-1)+d+1) &= \mfc_{d-1} \lp  \omega(d,1)(m-d-1) \rp^{(d-3)!},\\
		&= \mfc_{d-1} \lp \frac{1}{(d-1)!}(m-d-1)^{d-2} \rp^{(d-3)!},\\
		&= \mfc_{d-1} \lp \frac{1}{(d-1)!} \rp ^{(d-3)!} \lp (m-d-1)^{d-2} \rp^{(d-3)!},\\
		&= \mfc_d(m-d-1)^{(d-2)!}.
	\end{align*}
	
	\noindent
	Consequently,
	\begin{equation*}
	\Omega(d,m) = \Omega(d-1, \omega(d,1)(m-d-1)+d+1) = \mfc_d(m-d-1)^{(d-2)!}.
	\end{equation*}
\end{proof}

In the following lemma, we prove an inequality that will be useful in the proof of Proposition \ref{prop:TheOmegaCondition}, the proposition which establishes the criterion we seek.

\begin{lemma}\label{lem:BoundingMfcd} \textbf{(Bounding $\log\lp \mfc_d \rp$)}\\
	For each $d \geq 4$,
	\begin{equation*}
		\log\lp \mfc_d \rp \geq 2(d-2)! - 2(d-2)!\log(d-1) - 2(d-3)!\log(d-1).
	\end{equation*}
\end{lemma}

\begin{proof}
	Observe
	\begin{equation*}
	\log\lp \mfc_d \rp = -\SL_{i=3}^{d-1} (i-2)!\log(i!) \geq - \SL_{i=3}^{d-1} (i-2)!\log\lp e^{1-i}i^i \rp,
	\end{equation*}
	
	\noindent
	where the approximation is due to Stirling's approximation (equation \ref{eqn:Stirling's Approximation}). Hence,
	\begin{align*}
		\log\lp \mfc_d \rp &\geq - \SL_{i=3}^{d-1} (i-2)!\lp 1-i + i\log(i) \rp,\\
		&= \SL_{i=3}^{d-1} (i-1)! - \SL_{i=3}^{d-1} (i-2)!(i)\log(i),\\
		&= \SL_{j=1}^{d-3} (j+1)! - \SL_{j=1}^{d-3} j!(j+2)\log(j+2),\\
		&= \SL_{j=1}^{d-3} (j+1)! - \SL_{j=1}^{d-3} j!(j+1)\log(j+2) - \SL_{j=1}^{d-3} j!\log(j+2),\\
		&\geq \SL_{j=1}^{d-3} (j+1)! - \log(d-1)\SL_{j=1}^{d-3} (j+1)! - \log(d-1)SL_{j=1}^{d-3} j!,\\
		&= (1-\log(d-1)) \SL_{j=1}^{d-3} (j+1)! - \log(d-1)\SL_{j=1}^{d-3} j!. 
	\end{align*}
	
	\noindent
	Recall that for any positive integer $a$, $\SL_{i=1}^a i! \leq 2(a!)$. Since $d \geq 4$, it follows that $1-\log(d-1)$ and $-\log(d-1)$ are negative, hence
	\begin{align*}
		\log\lp \mfc_d \rp &\geq 2(1-\log(d-1))(d-2)! - 2\log(d-1)(d-3)!,\\
		&= 2(d-2)! - 2(d-2)!\log(d-1) - 2(d-3)!\log(d-1).
	\end{align*}
\end{proof}

\begin{prop}\label{prop:TheOmegaCondition} \textbf{(The $\Omega$ Condition)}\\
	Fix $d \geq 6$. For any $m \geq d^2-d+4$ such that
	\begin{align*}
		m^2 - \frac{5}{2}m + \frac{1}{2} < (d+1) + \log(d)\lp d+ \frac{1}{2} \rp + 6(d-3)!(d-2-\log(d-1)),
	\end{align*}
	
	\noindent
	it follows that
	\begin{align*}
		\Phi(d,m-d-1) < \Phi(d+1,m-d-2).
	\end{align*}
\end{prop}

\begin{proof} In light of Lemma \ref{lem:Psi(d,k)d-2NonDecreasing}, it suffices to have
	\begin{equation}\label{eqn:omega1}
		\frac{(m-1)!}{d!} < \dim\lp \mcm(3,\psi(d+1,m-d-2)_{d-1} \rp + m.
	\end{equation}
	
	\noindent
	First, we approximate $\psi(d+1,m-d-2)_{d-1}$ below by $\left\lceil \Omega(d+1,m-d-2) \right\rceil$ to get
	\begin{equation}
	\dim\lp \mcm\lp 3,\psi(d+1,m-d-2)_{d-1} \rp \rp + m \\
		\geq \dim\lp \mcm\lp 3,\left\lceil \Omega(d+1,m-d-2) \right\rceil \rp \rp + m. 
	\end{equation}
	
	\noindent
	Observe that
	\begin{align*}
		&\dim\lp \mcm\lp 3,\left\lceil \Omega(d+1,m-d-2) \right\rceil \rp \rp + m \\
		&= \frac{1}{6}\lp \left\lceil \Omega(d+1,m-d-2) \right\rceil^3 + 6\left\lceil \Omega(d+1,m-d-2) \right\rceil^2 + 11\left\lceil \Omega(d+1,m-d-2) \right\rceil + 6 \rp\\
		&\geq \frac{1}{6}\lp \Omega(d+1,m-d-2)^3 + 6\Omega(d+1,m-d-2)^2 + 11\Omega(d+1,m-d-2) + 6 \rp\\
		& \ - \lp \Omega(d+1,m-d-2)+1 \rp^2 + m,\\
		&= \frac{1}{6}\Omega(d+1,m-d-2)^3 - \frac{1}{6}\Omega(d+1,m-d-2)+m.
	\end{align*}
	
	\noindent
	Since $d \geq 6$ and $m \geq d^2-d+4 \geq 34$, it follows that
	\begin{align*}
		\Omega(d+1,m-d-2) \geq \Omega(6,m-8) \geq \Omega(6,26) > 5.6 \times 10^{12},
	\end{align*}
	
	\noindent
	and consequently
	\begin{equation}\label{eqn:omega2}
		\frac{1}{6}\Omega(d+1,m-d-2)^3 - \frac{1}{6}\Omega(d+1,m-d-2)+m > \frac{1}{7}\Omega(d+1,m-d-2)^3.
	\end{equation}
	
	\noindent
	Substituting inequality (\ref{eqn:omega2}) into inequality (\ref{eqn:omega1}) and re-arranging, we have the sufficient criterion
	\begin{align*}
		\frac{7(m-1)!}{d!} < \Omega(d+1,m-d-2)^3 = \mfc_d^3  (m-d-2)^{3(d-2)!}.
	\end{align*}
	
	\noindent
	Next, we apply Stirling's approximations (equation \ref{eqn:Stirling's Approximation})and obtain
	\begin{align*}
		\frac{ 7(m-1)^{m-\frac{1}{2}}e^{2-m} }{ \sqrt{2\pi}d^{d+\frac{1}{2}}e^{-d} } < \mfc_d^3(m-d-2)^{3(d-2)!},
	\end{align*}
	
	\noindent
	which we re-arrange as
	\begin{equation}\label{eqn:omega3}
		\frac{ (m-1)^{m-\frac{1}{2}} }{ e^m(m-d-2)^{3(d-2)!} } < \frac{ \sqrt{2\pi} }{ 7 } \cdot e^{d-2} \cdot d^{d+\frac{1}{2}} \cdot \mfc_d^3.
	\end{equation}
	
	We take $\log$ of both sides of inequality (\ref{eqn:omega3}) and examine them individually. First, the right side of inequality (\ref{eqn:omega3}) becomes
	\begin{equation}\label{eqn:omega4}
		\log\lp \frac{ \sqrt{2\pi} }{7} \rp + (d+2) + \lp d+\frac{1}{2} \rp\log(d) + 3\log\lp \mfc_d \rp.
	\end{equation}
	
	\noindent
	By applying Lemma \ref{lem:BoundingMfcd} to $\log\lp \mfc_d \rp$, observing $\log\lp \frac{\sqrt{2\pi}}{7} \rp \geq -1$, and simplifying, it suffices to replace expression (\ref{eqn:omega4}) with
	\begin{equation}\label{eqn:omega5}
		(d+1) +  \lp d+\frac{1}{2} \rp\log(d) + 6(d-2)! - 6(d-2)!\log(d-1) - 6(d-3)!\log(d-1).
	\end{equation}
	
	\noindent
	The left side of inequality (\ref{eqn:omega3}) becomes
	\begin{equation}\label{eqn:omega6}
		\lp m - \frac{1}{2} \rp \log(m-1) - m - 3(d-2)!\log(m-d-2).
	\end{equation}
	
	\noindent
	For $m \geq d^2-d+4$,
	\begin{equation*}
		\log(m-d-2) > \log\lp d^2-2d+1 \rp = 2\log(d-1),
	\end{equation*}
	
	\noindent
	and multiplying by $3(d-2)!$ yields
	\begin{equation}\label{eqn:omega7}
		3(d-2)!\log(m-d-2) > 6(d-2)!\log(d-1).
	\end{equation}
	
	By combining expressions (\ref{eqn:omega5}) and (\ref{eqn:omega6}) with inequality (\ref{eqn:omega7}), we obtain
	\begin{align*}
		\lp m - \frac{1}{2} \rp \log(m-1) - m < (d+1) + \lp d+\frac{1}{2} \rp\log(d) + 6(d-2)! - 6(d-3)!\log(d).
	\end{align*}
	
	Using the approximation $\log(m-1) < m-1$, we finally arrive at the condition
	\begin{align*}
		m^2 - \frac{5}{2}m + \frac{1}{2} < (d+1) + \lp d+\frac{1}{2} \rp\log(d) + 6(d-3)! \lp d-2 - \log(d-1) \rp,
	\end{align*}
	
	\noindent
	as claimed above.
\end{proof}

Using the simple approximations $\log(d) \geq \log(6) > 1$ and $d-2-\log(d-1) > 1$, we arrive at a simplified condition.

\begin{corollary}\label{cor:TheSimplifiedOmegaCondition} \textbf{(The Simplified $\Omega$ Condition)}\\
	Fix $d \geq 6$. For any $m \geq d^2-d+4$ such that
	\begin{align*}
		m^2 - \frac{5}{2}m \leq 6(d-3)! + 2d + 1,
	\end{align*}
	
	\noindent
	it follows that
	\begin{align*}
		\Phi(d,m-d-1) < \Phi(d+1,m-d-2).
	\end{align*}
	
	\noindent
	Moreover,
	\begin{align*}
		F(m) \geq \frac{(m-1)!}{d!}.
	\end{align*}
\end{corollary}

\begin{proof} Together, Lemma \ref{lem:Psi(d,k)d-2NonDecreasing} and Proposition \ref{prop:TheOmegaCondition} yield that
	\begin{align*}
		\Phi(d,m-d-1) < \Phi(d',m-d'-1)
	\end{align*}
	
	\noindent
	for each $d < d' < m-2$. For any $d'' < d$, we have
	\begin{align*}
		\frac{(m-1)!}{d!} < \frac{(m-1)!}{(d'')!} \leq \Phi(d'',m-d''-1).
	\end{align*}
	
	\noindent
	Hence,
	\begin{align*}
		F(m) \geq 2\left\lfloor \frac{1}{2}\Phi(d,m-d-1) \right\rfloor+1 \geq \frac{(m-1)!}{d!}.
	\end{align*}
\end{proof}

\noindent
We state and prove a corollary, then we recall and prove Theorem \ref{thm:ComparisonWithF}.

\begin{corollary}\label{cor:Bounding the Ratio F(m)/G(m)}  \textbf{(Bounding the Ratio $\frac{F(m)}{G(m)})$}\\
For $d \geq 11$ and $m \geq 2d^2+11d+15$, $	\frac{F(m)}{G(m)} > d+1.$
\end{corollary}

\begin{remark}
We expect that better estimates of $\frac{F(m)}{G(m)}$ could reasonably be obtained. However, Corollary \ref{cor:Bounding the Ratio F(m)/G(m)} suffices to prove Theorem \ref{thm:ComparisonWithF}, which establishes that $G(m)$ is the better bounding function and thus we do not need additional data on the growth rate of $F(m)$.
\end{remark}

\begin{proof} \textbf{(Proof of Corollary \ref{cor:Bounding the Ratio F(m)/G(m)})}\\
Let $d \geq 11$. Recall that Corollary \ref{cor:UpperBoundOnGrowthRateofG} applies for $m \geq 2d^2+7d+6$ and so we set $m_d = 2d^2 + 7d + 6$. Similarly, Corollary \ref{cor:TheSimplifiedOmegaCondition} applies for $m \geq d^2-d+4$ such that
	\begin{align*}
		m^2 - \frac{5}{2}m \leq 6(d-3)!+2d+1.	
	\end{align*}
	
	\noindent
	Correspondingly, we set
	\begin{align*}
		M_d = \max\lb m \in \ZZ \st m^2 - \frac{5}{2}m \leq 6(d-3)!+2d+1 \rb.
	\end{align*}
	
	\noindent
	 Observe that $2d^2+7d+6 \geq d^2-d+4$ and, since $d \geq 11$, we have
	\begin{align*}
		m_{d+1} = 2d^2+11d+15 < M_d.
	\end{align*}		 
	 
	Corollary \ref{cor:UpperBoundOnGrowthRateofG} yields $G(m_{d+1}) < \frac{(m_d-1)!}{(d+1)!}$ and Corollary \ref{cor:TheSimplifiedOmegaCondition} yields that $F(m_{d+1}) \geq \frac{(m_d-1)!}{d!}$. As a result, we have
	\begin{align*}
		\frac{F(m_{d+1})}{G(m_{d+1})} > \frac{ \lp \frac{(m_d-1)!}{d!} \rp }{ \lp \frac{(m_d-1)!}{(d+1)!} \rp} = \frac{(d+1)!}{d!} = d+1.
	\end{align*}
	
	\noindent
	In fact,	
	\begin{align*}
		\lp 6(d-3)! + 2d + 1 \rp - \lp m_{d+2}^2 - \frac{5}{2}m_{d+2} \rp
	\end{align*}
	
	\noindent
	is positive and strictly increasing for $d \geq 11$, so $m_{d+2} < M_d$.  Hence, Corollaries \ref{cor:UpperBoundOnGrowthRateofG} and \ref{cor:TheSimplifiedOmegaCondition} yield that
	\begin{align*}
		\frac{F(m)}{G(m)} > \frac{ \lp \frac{(m-1)!}{d!} \rp }{ \lp \frac{(m-1)!}{(d+1)!} \rp} = d+1.
	\end{align*}
	
	\noindent
	for all $m \geq m_d = 2d^2 + 7d + 6$.
\end{proof}

\begin{proof} \textbf{(Proof of Theorem \ref{thm:ComparisonWithF})}\\
	First, observe that Corollary \ref{cor:Bounding the Ratio F(m)/G(m)} implies that
	\begin{align*}
		\lim\limits_{m \ra \infty} \frac{F(m)}{G(m)} = \infty,
	\end{align*}	
	
	\noindent
	and that $G(m) \leq F(m)$ for
	\begin{align*}
		m \geq 2\lp 11^2 \rp + 11(11) + 15 = 378.
	\end{align*}
	
	The verification of that $G(m) \leq F(m)$ for $1 \leq m \leq 59$ comes from explicit computation and the relevant data is provided in Appendices \ref{subsec:Appendix - Explicit Bounds} and \ref{subsec:Appendix - Explicit Approximations of F(m)/G(m)}.
	
	Finally, we address the cases of $60 \leq m \leq 377$. Recall that Lemma \ref{lem:UpperBoundOnVartheta} yields
	\begin{align*}
		\vartheta(d,m-d-1) \leq m-d-2 + \binom{m}{d}.
	\end{align*}
	
	For a fixed $d$, $\vartheta(m,d-m-1)$ is bounded above by a polynomial of degree $d$ in $m$ with positive coefficients. Remark \ref{rem:DimensionOfParameterAndModuliSpaces} yields that $\dim\lp \mcm(2,\dotsc,d;\vartheta(d,m-d-1) \rp$ is also a polynomial in $m$, hence there is a polynomial $p_d(m)$ with positive coefficients that bounds $\dim\lp \mcm(2,\dotsc,d;\vartheta(d,m-d-1) \rp$ above. Hence, there is a minimal positive integer $a_d$ such that $\frac{(a_d-1)!}{d!} > p_d(a_d)$. Moreover, $\frac{m!}{d!} > p_d(m)$ for all $m \geq a_d$. In particular,
	\begin{align*}
		G(m) \leq \frac{(m-1)!}{d!}
	\end{align*}
	
	\noindent
	for all $m \geq a_d$. \\
	
	We compute explicitly that for all $m \geq 57$, $\frac{(m-1)!}{9!} > \dim\lp \mcm(2,\dotsc,9;\vartheta(9,m-10) \rp$ and thus
	\begin{align*}
		G(m) \leq \frac{(m-1)!}{9!}.
	\end{align*}
	
	\noindent
	Additionally, we explicitly compute that for $m \leq 377$,
	\begin{align*}
		\frac{(m-1)!}{6!} < \dim\lp \mcm(3;\psi(7,m-d-1)_{5} \rp.
	\end{align*}
	
	\noindent
	and the same argument proving Corollary \ref{cor:TheSimplifiedOmegaCondition} yields that
	\begin{align*}
		F(m) \geq \frac{(m-1)!}{6!}.
	\end{align*}
	
	\noindent
	As a consequence,
	\begin{align*}
		\frac{F(m)}{G(m)} \geq \frac{ \frac{(m-1)!}{6!} }{ \frac{(m-1)!}{9!} } \geq \frac{9!}{6!} = 504
	\end{align*}
	
	\noindent
	for all $60 \leq m \leq 376$, which yields the theorem.	 
\end{proof}


\section{Appendices}\label{sec:Appendices}

\subsection{Explicit Bounds on $\RD(n)$}\label{subsec:Appendix - Explicit Bounds}

Here we provide initial data on the behavior of $G(m)$ and provide data about $F(m)$ for comparison.

\begin{center}
Table 1: Upper Bounds on $\RD(n)$
\begin{tabular}{|c|c|c|c|c|}
\hline
$m$ &$G(m)$ &$F(m)$ &$F(m)/G(m)$ &First Established by\\
\hline
2 &3 &3 &1 &Babylonians \& Egyptians\\
3 &4 &4 &1 &Ferrari\\
4 &5 &5 &1 &Bring, in \cite{Bring1786}\\
5 &9 &9 &1 &Segre, in \cite{Segre1945}\\
\hline
6 &21 &41 &1.952 &Theorem \ref{thm:The n-6 Bound}, fixing the gap in the proof in \cite{Chebotarev1954}\\
\hline
7 &109 &121 &1.175 &\\
8 &325 &841 &2.645 &\\
9 &1681 &6721 & 3.998 &Theorem \ref{thm:The n-7,...,n-14 Bounds}\\
10 &15121 &60481 &4.000 &\\
11 &151,201 &604,801 &4.000 &\\
\hline
12 &1,663,201 &6,652,801 &4.000 &\\
13 &19,958,401 &78,485,043 &3.932 &Theorem \ref{thm:The n-7,...,n-14 Bounds} / Theorem \ref{thm:Determining A Point On Tau(d+k)}\\
14 &259,459,201 &320,082,459 &1.234 &\\
\hline
15 &3,632,428,801 &3,632,428,801 &1 &Wolfson, in \cite{Wolfson2021}\\
16 &54,486,432,001 &54,486,432,001 &1 &\\
\hline
17 &348,489,068,134 &871,782,912,001 &2.502 &Theorem \ref{thm:Determining A Point On Tau(d+k)}\\
18 &2,964,061,900,801 &14,820,309,504,001 &5 &\\
\hline
\end{tabular}
\end{center}

\noindent
When $m=12, 13, 14$, the methods of Theorems \ref{thm:The n-7,...,n-14 Bounds} and \ref{thm:Determining A Point On Tau(d+k)} both obtain the new bound on resolvent degree.

\newpage

\subsection{Explicit Approximations of $F(m)/G(m)$}\label{subsec:Appendix - Explicit Approximations of F(m)/G(m)}

Here we provide additional data about the ratio $\frac{F(m)}{G(m)}$. In particular, the behavior exhibited from $m=58$ to $m=59$ illustrates why the ratio $\frac{F(m)}{G(m)}$ is not always non-decreasing.

\begin{center}
Table 2: $F(m)/G(m)$ for $19 \leq m \leq 59$\\
\begin{tabular}{|c|c|c|c|}
\hline
$m$ &$F(m)/G(m)$ &$G(m)$ given by determining an &$F(m)$ given by determining an\\
\hline
19 &5.000 &  &\\
20 &5.000 &  &\\
21 &5.000 &$(m-6)$-plane on $\tau_{1,\dotsc,5}$ &$(m-5)$-plane on $\tau_{1,2,3,4}$\\
22 &5.000 &  &\\
23 &5.000 &  &\\
24 &5.000 &  &\\
\hline
25 &29.930  &  &\\
26 &30.000 & &\\
27 &30.000 & &\\
28 &30.000 & &\\
29 &30.000 &$(m-7)$-plane on $\tau_{1,\dotsc,6}$ &$(m-5)$-plane on $\tau_{1,2,3,4}$\\
30 &30.000 & &\\
31 &30.000 & &\\
32 &30.000 & &\\
33 &30.000 & &\\
\hline
34 &146.129 & &\\
35 &210.000 & &\\
36 &210.000 & &\\
37 &210.000 & &\\
38 &210.000 &$(m-8)$-plane on $\tau_{1,\dotsc,7}$ &$(m-5)$-plane on $\tau_{1,2,3,4}$\\
39 &210.000 & &\\
40 &210.000 & &\\
41 &210.000 & &\\
42 &210.000 & &\\
43 &210.000 & &\\
\hline
44 &294.103 & &\\
45 &1680.000 & &\\
46 &1680.000 & &\\
47 &1680.000 & &\\
48 &1680.000 & &\\
49 &1680.000 &$(m-9)$-plane on $\tau_{1,\dotsc,8}$ &$(m-5)$-plane on $\tau_{1,2,3,4}$\\
50 &1680.000 & &\\
51 &1680.000 & &\\
52 &1680.000 & &\\
53 &1680.000 & &\\
54 &1680.000 & &\\
55 &1680.000 & &\\
\hline
56 &2613.173 & &\\
57 &15120.000 &$(m-10)$-plane on $\tau_{1,\dotsc,9}$ &$(m-5)$-plane on $\tau_{1,2,3,4}$\\
58 &15120.000 & &\\
\hline
59 &3024.000 &$(m-10)$-plane on $\tau_{1,\dotsc,9}$ &$(m-6)$-plane on $\tau_{1,\dotsc,5}$\\
\hline
\end{tabular}
\end{center}


\newpage

\subsection{Proof of Technical Lemma}\label{subsec:Appendix-ProofOfTechnicalLemma}

We now recall the statement of Lemma \ref{lem:TechnicalLemma} and provide a proof. We will continue to use the notation established in Definition \ref{def:PolarsOfAPolynomial}. \\

\noindent
\textbf{Lemma \ref{lem:TechnicalLemma}. (Technical Lemma)}\\
Let $P,Q \in \PP^r(K)$ and $f \in K[x_0,\dotsc,x_r]$ be a homogeneous polynomial of degree $d$. Applying a projective change of coordinates as necessary, we assume that 
	\begin{align*}
		P &= [1:p_1:\cdots:p_r],\\
		Q &= [1:q_1:\cdots:q_r],
	\end{align*}		
	
	\noindent
	so that the line determined by $P$ and $Q$ is
	\begin{align*}
		L(P,Q)(K) = \lb [1:\lambda p_1 + \mu q_1:\cdots:\lambda p_r + \mu q_r] \st [\lambda:\mu] \in \PP^1(K) \rb.
	\end{align*}
	
	\noindent
	For any point $R_{\lambda:\mu} = [1:\lambda p_1 + \mu q_1:\cdots:\lambda p_r + \mu q_r] \in L(P,Q)(K)$,
	\begin{equation*}
		f\lp R_{\lambda:\mu} \rp = f(\lambda P) + f(\mu Q) + \SL_{k=1}^{d-1} \frac{1}{k!} t(d-k,f,\lambda P)(\mu Q).
	\end{equation*}

\begin{proof}
	Set $p_0=1$ and $q_0=1$. Let $f \in K[x_0,\dotsc,x_r]$ be a homogeneous polynomial of degree $d$. When $r=1$, the claim follows immediately from the binomial formula. We thus assume $r \geq 2$. As partial derivatives are linear, it suffices to consider the case where $f$ is a monomial:
	\begin{align*}
		f(x_0,\dotsc,x_r) = a x_0^{i_0} x_1^{i_1} \cdots x_r^{i_r}.
	\end{align*}
	
	Note that
	\begin{align*}
		f\lp R_{\lambda:\mu} \rp &= a \PL_{j=0}^r \lp \lambda p_j + \mu q_j \rp^{i_j}\\
		&= a\PL_{j=0}^r \SL_{\ell_j=0}^{i_j} \binom{i_j}{\ell_j} (\lambda p_j)^{\iota(j)-\ell_j} (\mu q_j)^{\ell_j}\\
		&= a \SL_{\ell_0=0}^{i_0} \cdots \SL_{\ell_r=0}^{i_r} \lp \PL_{j=0}^r \binom{i_j}{\ell_j} (\lambda p_j)^{i_j-\ell_j} (\mu q_j)^{\ell_j} \rp.
	\end{align*}
	
	To simplify notation, we denote an indexing set
	\begin{align*}
		I = \lb (k_0,\dotsc,k_r) \in \ZZ^{r+1} \st 0 \leq k_j \leq i_j \rb
	\end{align*}
	
	and, in accordance with the notation of Definition \ref{def:PolarsOfAPolynomial}, partition it into subsets
	\begin{align*}
		I_k = \lb (k_0,\dotsc,k_r) \in \ZZ^{r+1} \st k_0+\cdots+k_r=k \rb
	\end{align*}
	
	\noindent
	for $0 \leq k \leq d$. Thus, we write
	\begin{align*}
		f\lp R_{\lambda:\mu} \rp &= a \SL_{I} \lp \PL_{j=0}^r \binom{i_j}{k_j} (\lambda p_j)^{\iota(j)-k_j} (\mu q_j)^{k_j} \rp\\
		&= a \SL_{k=0}^d \SL_{I_k} \lp \PL_{j=0}^r \binom{i_j}{k_j} (\lambda p_j)^{i_j-k_j} (\mu q_j)^{k_j} \rp
	\end{align*}
	
	\noindent
	Note that
	\begin{align*}
		f(\lambda P) &=  a \lp \PL_{j=0}^r (\lambda p_j)^{i_j} \rp = a\SL_{I_0} \lp \PL_{j=0}^r \binom{i_j}{0} (\lambda p_j)^{i_j-0} (\mu q_j)^{0} \rp, \\
		\\
		f(\mu Q) &= a \lp \PL_{j=0}^r (\mu q_j)^{i_j} \rp = a\SL_{I_d} \lp \PL_{j=0}^r \binom{i_j}{i_j} (\lambda p_j)^{i_j-i_j} (\mu q_j)^{i_j} \rp.
	\end{align*}
	
	\noindent
	Thus,
	\begin{align*}
		f\lp R_{\lambda:\mu} \rp - f(\lambda P) - f(\mu Q)
		&= a\SL_{k=1}^{d-1} \SL_{I_k} \lp \PL_{j=0}^r \binom{i_j}{l_j} (\lambda p_j)^{i_j-k_j} (\mu q_j)^{k_j} \rp.
	\end{align*}
	
	\noindent
	Consequently, it suffices to show that
	\begin{align*}
		\frac{1}{k!} t(d-k,f,\lambda P)(\mu Q) = a\SL_{I_k} \lp \PL_{j=0}^r \binom{i_j}{l_j} (\lambda p_j)^{i_j-k_j} (\mu q_j)^{k_j} \rp
	\end{align*}
	
	\noindent
	for each $1 \leq k \leq d-1$. Recall that
	\begin{align*}
	[k]^* &= \lb 1,\dotsc,k \rb,\\
	[r] &= \lb 0,1,\dotsc,r \rb,\\
	I_k^* &= \Hom_{\Set}([k]^*,[r])
	\end{align*}
	
	\noindent	
	and for $\iota \in I_k^*$, we set $|\iota|(j) = |\iota\inv(j)|$. Hence,
	\begin{align*}
		& \ \frac{1}{k!} t(d-k,f,\lambda P)(\mu Q) \\
		&= \frac{1}{k!} \SL_{\iota \in I_k^*} \lp \p_0^{|\iota|(0)} \cdots \p_r^{|\iota|(0)} f \rp \biggr\rvert_{\lambda P} \  (\mu q_0)^{|\iota|(0)}\cdots (\mu q_r)^{|\iota(r)|}\\
		&= \frac{1}{k!} \SL_{\iota \in I_k^*} a\frac{i_0}{(i_0-|\iota|(0))!} \cdots \frac{i_r}{(i_r - |\iota|(r))!} (\lambda p_0)^{i_0-|\iota|(0)} \cdots (\lambda p_r)^{i_r-|\iota|(r)}  (\mu q_0)^{|\iota|(0)}\cdots (\mu q_r)^{|\iota(r)|}\\
		&= \frac{a}{k!} \SL_{\iota \in I_k^*} \PL_{j=0}^r \frac{i_j}{(i_j-|\iota|(j))!}(\lambda p_j)^{i_j-|\iota|(j)} (\mu q_j)^{|\iota|(j)}
	\end{align*}
	
	\noindent
	Note that for any $\iota,\iota' \in I_k^*$ whose fibers at every point of $[r]$ have the same cardinality, we have
	\begin{align*}
		\p_0^{|\iota|(0)} \cdots \p_r^{|\iota|(0)} f = \p_0^{|\iota'|(0)} \cdots \p_r^{|\iota'|(0)} f
	\end{align*}
	
	\noindent
	as partial derivatives of polynomials commute. Now, given any $\iota \in I_k^*$, we can apply a permutation of the symmetric group $S_k$ to $[k]^*$ and not change the number of fibers of each cardinality. However, to get a unique representative of each unordered class of partial derivative, we must identify any permutations which fix all of the fibers, but permute the elements within any such fiber. Thus, to a given $\iota$, there are
	\begin{align*}
		\frac{\lav S_k \rav}{\lav S_{|\iota|(0)} \rav \cdots \lav S_{|\iota(r)|} \rav} = k! \PL_{l=0}^r \frac{1}{(|\iota|(j))!}
	\end{align*}
	
	\noindent
	permutations which give the same unordered partial derivative. Hence,
	\begin{align*}
	\frac{1}{k!} t(d-k,f,\lambda P)(\mu Q) &= \frac{a}{k!} \SL_{\iota \in I_k^*} \PL_{j=0}^r \frac{i_j}{(i_j-|\iota|(j))!}(\lambda p_j)^{i_j-|\iota|(j)} (\mu q_j)^{|\iota|(j)}\\
	&= \frac{a}{k!} \SL_{I_k} \lp k! \PL_{l=0}^r \frac{1}{(|\iota|(j))!} \rp \PL_{j=0}^r \frac{i_j}{(i_j-k_j)!}(\lambda p_j)^{i_j-k_j} (\mu q_j)^{k_j}\\
	&= a \SL_{I_k} \PL_{j=0}^r \binom{i_j}{k_j}(\lambda p_j)^{i_j-k_j} (\mu q_j)^{k_j}
	\end{align*}
	
	\noindent
	which yields the lemma.
\end{proof}



\end{document}